\documentclass[11pt]{article}
 
\usepackage{amsmath,amsthm,amssymb}
\usepackage[T1]{fontenc}
\usepackage[utf8]{inputenc}
\usepackage[dvipdf]{graphicx}
\usepackage{color}
\usepackage{epstopdf}
\usepackage{dsfont}
\usepackage{enumerate}
\usepackage{enumitem}

\definecolor{db}{RGB}{0, 0, 130}
\usepackage[colorlinks=true,citecolor=red,linkcolor=db,urlcolor=blue,pdfstartview=FitH]{hyperref}
 
\usepackage[normalem]{ulem}
\usepackage[numbers]{natbib}

\definecolor{rp}{rgb}{0.25, 0, 0.75}
\definecolor{dg}{rgb}{0, 0.6, 0}

\newtheorem{theorem}{Theorem}[section]

\newtheorem{definition}{Definition}[section]

\newtheorem{assumption}[theorem]{Assumption}
\newtheorem{lemma}[definition]{Lemma}

\newtheorem{proposition}[definition]{Proposition}
\newtheorem{remark}[definition]{Remark}
\def\1{\mathbf{1}}
\def\R{\mathbb{R}}
\def\C{\mathbb{C}}
\def\D{\mathbb{D}}
\def\E{\mathbb{E}}

\def\Q{\mathbb{Q}}

\def\L{\mathbb{L}}
\def\F{\mathbb{F}}
\def\P{\mathbb{P}}
\def\S{\mathbb{S}}
\def\H{\mathbb{H}}
\def\G{\mathbb{G}}
\def\M{\mathbb{M}}
 
\def\Ac{\mathcal{A}}
\def\Bc{\mathcal{B}}
\def\Cc{\mathcal{C}}
\def\Fc{\mathcal{F}}
\def\Xb{\overline{X}}
\def\Xh{\widehat{X}}
\def\Xt{\widetilde{X}}
\def\Yt{\widetilde{Y}}
\def\Pb{\widehat{\P}}
\def\Yh{\widehat{Y}}

\def\Gc{\mathcal{G}}
\def\Lc{\mathcal{L}}
\def\Pc{\mathcal{P}}
\def\Qc{\mathcal{Q}}

\def\Kc{\mathcal{K}}
 
\def\Ut{\widetilde U}
\def\xit{\tilde \xi}
\def\Wt{\widetilde W}
\def\Lt{\widetilde L}
\def\Gt{\widetilde G}
 
\def\d{\mathrm{d}}

\def\xb{\mathbf{x}}
\def\yb{\mathbf{y}}
\def\wb{\mathbf{w}}
\def\bb{\mathbf{b}}
 
\def\Et{\widetilde \E}
 
\def\x{\times}
\def\Om{\Omega}
\def\om{\omega}
\def\eps{\varepsilon}
 
\def\Omh{\widehat{\Om}}
\def\Omb{\overline \Om}

\def\omb{\bar \om}

\def\Fcb{\overline \Fc}
\def\Pcb{\overline \Pc}
\def\Pb{\overline \P}
\def\Qb{\overline \Q}
\def\Vb{\overline V}
\def\Fb{\overline \F}

\def\Wc{\mathcal{W}}

\def\Sb{\overline S}
\def\Wb{\overline W}
\def\alphab{\overline \alpha}
 
\def\Lambdah{\widehat{\Lambda}}
\def\Wh{\widehat{W}}
\def\Fh{\widehat{\F}}
\def\Fch{\widehat{\Fc}}

\def\Gb{\overline \G}
\def\Gcb{\overline \Gc}

\def\mub{\bar \mu}
\def\muh{\widehat \mu}
\def\nub{\bar \nu}
\def\nuh{\widehat \nu}

\title{Limit theory for mean-field control problems\\ with common noise {adapted} controls}
 
\author{
Bruno Bouchard
\footnote{CEREMADE, Universit\'e Paris-Dauphine, PSL, CNRS. bouchard@ceremade.dauphine.fr. }
\and
Xiaolu Tan
\footnote{Department of Mathematics, The Chinese University of Hong Kong. xiaolu.tan@cuhk.edu.hk.}
}
 
\date{\today}

\begin{document}

\maketitle

\begin{abstract}
We consider a mean-field control problem in which   admissible controls are required to be adapted to the common noise filtration.
The main objective is to show how the mean-field control problem can be approximates by time consistent centralized  finite population problems in which the central planner has full information on all agents' states and gives an identical signal to all agents. We also aim at establishing the optimal convergence rate.
In a first general path-dependent setting, we only prove convergence  by using weak convergence techniques of probability measures on the canonical space.
Next, when only the drift coefficient is controlled, we obtain a backward SDE characterization of the value process,
based on which a convergence rate is established in terms of the Wasserstein distance between the original measure  and the empirical one  induced by the particles.
It requires  Lipschitz continuity conditions in the Wasserstein sense. The convergence rate is  optimal.
In a Markovian setting and under convexity conditions on the running reward function, we next prove uniqueness of the optimal control and provide regularity results on the value function, and then deduce the optimal weak convergence rate in terms of the number of particles.
Finally, we apply these results to the  study of a classical optimal control problem with partial observation,
leading to an original approximation method by particle systems.
 
\end{abstract}

\section{Introduction}

Mean-field control problems {consist in controlling  mean-field (or McKean-Vlasov) SDEs} so as to optimize a reward function.
These problems have been largely studied in the recent literature, motivated by their applications in economics, engineering, finance, etc.,   to model the limit behavior of  large population control problems.
Besides the study of classical dynamic programming and  maximum principle approaches,
an interesting and important subject is to establish the so-called propagation of chaos results,
that is, the convergence (rate) of  finite population control problems to a mean-field limit when the population's size $N$ goes to infinity.
General convergence results are obtained in Lacker \cite{LackerLimit} and Djete, Possama\"i and Tan \cite{DjeteApprox} by using weak convergence techniques.
A more recent stream of literature focuses on the convergence rate, strong or weak, 
{see \cite{BCC, BEZ, GMS, GPW, CardSoug, CardDJS} and in particular \cite{DaudDelarue, CardSharp} for the optimal convergence rate.}

\vspace{0.5em}
 
Motivated by its applications in  optimal control  with partial observation and   centralized optimal control problems,
we study in this paper a mean-field control problem with common noise,
in which the controlled dynamic is defined as the solution to a McKean-Vlasov SDE:
$$
X^{\alpha}_t
=
X_0
+ \int_0^t b(s, X^{\alpha}_s, \mu^{\alpha}_s, \alpha_s) \d s
+ \int_0^t \sigma(s, X^{\alpha}_s, \mu^{\alpha}_s, \alpha_s) \d W_s
+ \int_0^t \sigma_0 (s, X^{\alpha}_s, \mu^{\alpha}_s, \alpha_s) \d B_s,
$$
in which $W$ is the individual noise, {$B$ is} the common noise, $\mu^{\alpha}_t := \Lc(X^{\alpha}_t | \Fc^B_t)$ is the marginal law at $t$ of the controlled process given the common noise,
and the admissible control process $\alpha$ is required to be adapted to the filtration $\F^B = (\Fc^B_t)_{t \le T}$ generated by the common noise { $B$}.
The optimal control problem is   given by
\begin{equation} \label{eq:ctrl_pb_intro}
\sup\left\{ \E \Big[ \int_0^T L(t, \mu^{\alpha}_t, \alpha_t) dt + g(\mu^{\alpha}_T) \Big] : \alpha \mbox{ is }\F^B-\mbox{adapted}\right\}.
\end{equation}
{In the above, the coefficients could also depend on $X^\alpha$, see Remark \ref{rem:reward_coeff} below.}
 
As finite population approximation,
we consider a centralized control problem with agents indexed by $k = 1, \ldots, N$,
in which the central planner has full information on the states of all agents
and gives the identical control signal $\alpha$ to all agents.
The dynamic of agent $k = 1, \ldots, N$ is defined by
\begin{equation} \label{eq:X_N_intro}
dX^k_t
=
b(t, X^k_t, \mu^{N}_t, \alpha_t) \d t
+
\sigma(t, X^k_t, \mu^{N}_t, \alpha_t) \d W^k_t
+
\sigma_0(t, X^k_t, \mu^{N}_t, \alpha_t) \d B_t,
\end{equation}
where $(W^1, \ldots, W^N)$ is a Brownian motion independent of $B$ and  $\mu^{N} := \frac1N \sum_{k=1}^N \delta_{X^k}$ is the empirical measure process of the particles.
The corresponding  finite population control problem is then given by
\begin{equation} \label{eq:ctrl_pb_N_intro}
\sup\left\{ \E \Big[ \int_0^T L(t, \mu^N_t, \alpha_t) dt + g(\mu^N_T) \Big]: \alpha \mbox{ is }\F^N-\mbox{adapted} \right\},
\end{equation}
in which $\F^N$ is the filtration generated by $(W^1, \ldots, W^N, B)$.
{Let us mention again that our mean-field control problem is different from the classical ones since the control process $\alpha$ in \eqref{eq:ctrl_pb_intro} is required to be adapted to the common noise filtration $\F^B$, while   admissible controls in the finite population problems  \eqref{eq:ctrl_pb_N_intro} are   adapted to the full information  flow generated by $(W^1, \ldots, W^N, B)$ and is shared by all agents $k=1, \ldots, N$ in \eqref{eq:X_N_intro}. The main advantage of this (probably counter-intuitive) finite population formulation is that it makes it time-consistent, so that classical dynamic programming approaches can be used to solve \eqref{eq:ctrl_pb_N_intro} numerically, see (ii) of Remark \ref{rem : HJB eq} below. On the other hand, the fact that particles share the same control will imply that the observation of  $(W^1, \ldots, W^N)$ will not play any role at the limit $N\to \infty$.}
Such a centralized control problem also appears very naturally in the applications of the mean-field theory in social/economic problems,
where the policy maker has full information on the whole population and then applies a universal policy on each individual {(e.g.~taxe rule, etc.).}

\vspace{0.5em}
Our main objective is to prove the convergence of the finite population control problem \eqref{eq:ctrl_pb_N_intro}  to the mean-field control problem \eqref{eq:ctrl_pb_intro} as $N \longrightarrow \infty$, to establish a convergence rate, and also to study the regularity of the value function.
 
\vspace{0.5em}
 
We first consider a general path-dependent setting and prove convergence, by adapting  weak convergence techniques, see e.g.~\cite{LackerLimit} and \cite{DjeteApprox}, for classical mean-field control problems with and without common noise.
The main idea consists in considering an appropriate relaxed formulation on the canonical space,
such that the space of all relaxed controls is closed, convex and can be approximated by strong/weak controls.
Then, it is enough to show that the sequence of $\eps$-optimal controls in \eqref{eq:ctrl_pb_N_intro} is tight and that any convergent subsequence converges to a limit which belongs to the space of optimal relaxed controls for \eqref{eq:ctrl_pb_intro}.
 
\vspace{0.5em}
 
We then specialize to a path-dependent setting where only the drift coefficient is controlled,
which allows us to appeal to a Backward SDE (BSDE) approach to provide a convergence rate, {under Lipschitz-continuity conditions in the $\Wc_2$-Wasserstein sense.} It is optimal
 and  is indeed the same   as in the linear setting.
While the BSDE theory provides a strong tool for classical optimal control problems,
it has been seldomly used to study the value process of mean-field control problems.
One of the main reasons is that, for classical mean-field control problems, the individual noise disappears in the dependence of the distribution process in the limit case.
Note that mean-field BSDEs have also been introduced and studied, see e.g.~\cite{BLP} { among many others}. They are usually obtained as   auxiliary forward-backward systems in the stochastic maximum principle, but do not describe the value function itself.
Here, we actually characterize the value process (as well as the optimal control) of our mean-field control problem by a specific BSDE,
and then deduce the optimal convergence rate under a $\Wc_2$-Lipschitz condition by using standard stability analysis.
 
\vspace{0.5em}
 
We then specialize further to a Markovian setting, in which the value function can be characterized by a Hamilton-Jaobi-Bellman master equation,
and  look for a more analytic way of  establishing the optimal weak rate in the spirit of \cite{CardDJS, CardSharp, DaudDelarue}.
When the running reward function $L(\cdot,a)$ is strictly concave in the value of the control $a$, we show that there exists a unique weak optimal control  (defined as a probability measure on the canonical space) which is continuous w.r.t.~the initial distribution of the controlled process. Based on this, the $C^1$-regularity of the value function is deduced.
This approach is original and completely different from the one used in the literature for the regularity analysis of   classical mean-field control problems,
see e.g. \cite{CardSoug} for comparison.
Next, when the volatility coefficient   $\sigma$ and $\sigma_0$ are constants, we obtain a more precise characterization of the value function as the solution to a classical parabolic PDE parametrized by the initial measure, and then establish that it is indeed $C^2$.
Using classical arguments, this regularity result allows one to establish that the {weak} convergence rate is of order $1/N$, which is optimal.
 
\vspace{0.5em}
 
Finally, we note that the  \eqref{eq:ctrl_pb_intro} actually covers a class of optimal control problems with partial observation, as studied since the 80s,
using both the dynamic programming  and the maximum principle approaches, see e.g.~\cite{Bensoussan}.
While general explicit solutions  are not available,  except in the linear quadratic setting,
their numerical approximation has been seldomly investigated (see \cite{LTTang}).
Our convergence results provide a first step towards  a numerical approximation method for general partially observed control problems.
Indeed, as already highlighted,  the $N$-population control problem \eqref{eq:ctrl_pb_N_intro}
{is time-consistent and can therefore be solved by dynamic programming, recall  (ii) of Remark \ref{rem : HJB eq} below}.
 
\vspace{0.5em}
 
The rest of the paper is organized as follows.
In Section \ref{sec:mfc_form}, we introduce our mean-field control problem and its finite population approximation.
Section \ref{sec:path_depend} is dedicated to general convergence results in  the path-dependent setting.
We specialize to the Markovian case in  Section \ref{sec:Markovian}.
The application to optimal control problems under partial observation is discussed in Section \ref{sec:partial_observ}.
\vspace{0.5em}
 
\noindent {\bf Notations.}
 We collect here general notations that will be used all over the paper. Let $d \ge 1$ be a positive integer. We denote by $\S^d$ the space of all $d \x d$ matrices, and by $\Cc^d := C([0,T], \R^d)$ the space of all $\R^d$-valued continuous paths on $[0,T]$, $T>0$, which is a Polish space under the topology induced by 
{the sup-norm:  
$$
 	\|\xb\|:=\sup\{|\xb_t|,t\le T\},
	~\mbox{for any}~
	\xb = (\xb_t)_{t\le T}\in \Cc^d,
$$ 
in which $|\cdot|$ stands for the Euclidean norm on $\R^d$.}
Given a (nonempty) metric space $(E, \rho)$, we denote by $\Bc(E)$ its Borel $\sigma$-field,
and by $\Pc(E)$ the space of all associated Borel probability measures.
For $p \ge 1$, we denote by $\Pc_p(E)$ the space of all probability measures $\mu \in \Pc(E)$ such that {$\int_E \rho(e, e_0)^p \mu(\d e) < \infty$} for some fixed $e_0 \in E$.
The space $\Pc_p(E)$ is equipped with the Wasserstein distance $\Wc_p$.
Given $\mu \in \Pc(E)$ and a  $\mu$-integrable function $\varphi: E \longrightarrow \R$, we write
$$
\langle \mu, \varphi \rangle := \E^{\mu}[ \varphi ] := \int_E \varphi (e) \mu(\d e).
$$
We shall mainly work on the canonical space $\Om := \R^d \x \Cc^d \x \Cc^d$, with canonical element $(\xi, W, B)$, i.e. $\xi(\om) := \om^{\xi}$, $W_t(\om) := \om^W_t$ and $B_t(\om) := \om^B_t$, for all $\om = (\om^{\xi}, \om^B, \om^{W}) \in \Om$ and $t\le T$. Let $\F = (\Fc_t)_{t \le T}$ be the canonical filtration and $\G = (\Gc_t)_{t \le T}$ be the filtration generated by the process $B$, i.e.
$$
	\Fc_t := \sigma(\xi, W_s, B_s:~ s \le t),
	\quad 
	\Gc_t := \sigma(B_s:~ s \le t), 
	~\mbox{for all}~t\in [0,T].
$$
We let  $\Fc := \sigma(\xi, W, B)$.

\vspace{0.5em}

Given a probability measure $\P$, a random variable $Y$  and a sigma-algebra $\Sigma$ (defined on a suitable probability space), we denote by $\Lc^\P(Y)$ and  $\Lc^\P(Y|\Sigma)$ the law  of $Y$ and the law of $Y$ given $\Sigma$. When $\Sigma=\sigma(Z)$ for another random variable $Z$, we simply write $\Lc^\P(Y|Z)$.

\section{The mean-field control problem and its approximation}
\label{sec:mfc_form}
 
Let $A$ be a (non-empty) Polish space, under the metric $\rho$, and in which we fix a point $a_0 \in A$.
We consider the {(path-dependent)} drift and volatility coefficients 
$$
(b, \sigma, \sigma_0) : [0,T] \x \Cc^d \x \Pc_2(\Cc^d) \x A \longrightarrow \R^d \x \S^d \x \S^d,
$$
and the {(path-dependent)} reward functions
$$
L : [0,T] \x \Pc_2(\Cc^d) \longrightarrow \R
~~\mbox{and}~~
g: \Pc_2(\Cc^d) \longrightarrow \R.
$$
Throughout the paper, we fixe an initial distribution $\nu_0 \in \Pc_2(\R^d)$.

\subsection{The mean-field optimal control problem}
 
{ On the canonical space $(\Om, \Fc)$, we fix a probability measure  $\P$, 
under which $W$ and $B$ are standard Brownian motions, $\xi \sim \nu_0$ and $(\xi, W, B)$ are mutually independent.
We denote by}
\begin{equation} \label{eq:defAc0}
\Ac_0
~:=~
\Big\{
\alpha: [0,T] \x \Om \longrightarrow A ~: \alpha ~\mbox{is}~\G \mbox{-predictable and}~ \E \Big[ \int_0^T \rho(\alpha_t, a_0)^2 \d t \Big] < \infty
\Big\}.
\end{equation}
Then, for each $\alpha \in \Ac_0$, the controlled process $X^{\alpha}$ is defined { on $(\Om, \Fc, \P)$} by the (path-dependent) McKean-Vlasov SDE
\begin{equation} \label{eq:def_X_alpha}
X^{\alpha}_t
=
\xi
+ \int_0^t b(s, X^{\alpha}_{s \wedge \cdot}, \mu^{\alpha}_s, \alpha_s) \d s
+ \int_0^t \sigma(s, X^{\alpha}_{s \wedge \cdot}, \mu^{\alpha}_s, \alpha_s) \d W_s
+ \int_0^t \sigma_0 (s, X^{\alpha}_{s \wedge \cdot}, \mu^{\alpha}_s, \alpha_s) \d B_s,
~t\le T,
\end{equation}
where $X^{\alpha}_{t \wedge \cdot}$ denotes the process $X^{\alpha}$ with frozen path after time $t$ and $\mu^{\alpha}_t := \Lc(X^{\alpha}_{t \wedge \cdot} | B)$ is the law of $X^{\alpha}_{t \wedge \cdot}$ given the path of $B$.
We will assume standard Lipschitz conditions below to ensure that the SDE \eqref{eq:def_X_alpha} has a unique strong solution for every $\alpha \in \Ac_0$.
\vspace{0.5em}
Our mean-field optimal control problem
is defined by
\begin{equation} \label{eq:def_V}
V := \sup_{\alpha \in \Ac_0} J(\alpha),
~~\mbox{with}~
J(\alpha) := \E \Big[ \int_0^T L(t, \mu^{\alpha}_t, \alpha_t ) \d t + g(\mu^{\alpha}_T) \Big].
\end{equation}
Let us emphasize that   admissible controls in $\Ac_0$ are required to be adapted to the filtration {$\G$ generated by the common noise $B$,
while those of} classical mean-field control problems are adapted to the filtration $\F$ generated by both the individual and the common noise {$(W,B)$ and the initial condition $\xi$}.
 
\begin{remark} \label{rem:reward_coeff}
{In the literature, the reward} function is often written in terms of  the expected value of $g(X^{\alpha}_{{T \wedge \cdot}}, \mu^{\alpha}_T) $.
This can in fact be equivalently reduced to  a function of $\mu^{\alpha}_T$ only. Indeed, {by the definition of $\mu^{\alpha}_T$,}
$$
\E \big[ g \big( X^{\alpha}_{{T \wedge \cdot}}, \mu^{\alpha}_T \big) \big]
~=~
\E \big[ \big \langle g(\cdot, \mu^{\alpha}_T),~\mu^{\alpha}_T \big \rangle \big].
$$
The same applies to the {running reward term, as $\alpha$ is $\G$-adapted}.
\end{remark}

\subsection{The centralized control problems with finite population}
\label{subsec:NctrlPb}
 
As an approximation of our mean-field control problem, we introduce a sequence of finite population centralized control problems.
For each $N \ge 1$, let $(\Om^N, \Fc^N, \P^N)$ be a complete probability space equipped with a Brownian motion $B$ {(which can be taken to be the same as the one introduced in the preceding section)}, a sequence of random variables $(\xi_k)_{k = 1, \ldots, N}$ and a sequence of Brownian motions $(W^k)_{k = 1, \ldots, N}$,
where all the above random elements are mutually independent and {$\xi_k \sim \nu_k \in \Pc_2(\R^d)$} for each $k=1, \ldots, N$.
Let us denote by $\F^N = (\Fc^N_t)_{t \le T}$ the filtration generated by $(B, W^1, \ldots, W^N)$ and the initial random variables $(\xi^1, \ldots, \xi^N)$,
i.e.
$$
\Fc^N_t := \sigma(\xi^1, \ldots, \xi^N, B_s, W^1_s, \ldots, W^N_s~: s \le t ),
~~ t \ge 0.
$$
Let 
$$
\Ac_N
~:=~
\Big\{
\alpha: [0,T] \x \Om^N \longrightarrow A ~: \alpha ~\mbox{is}~\F^N \mbox{-predictable and}~ \E^{\P^N} \Big[ \int_0^T \rho(\alpha_t, a_0)^2 \d t \Big] < \infty
\Big\}
$$
 be the space of all admissible controls. 
Then, given   $\alpha \in \Ac_N$, the dynamics $(X^{N, k, \alpha})_{k\le N}$ of the $N$-population of agents is 
\begin{align} \label{eq:def_Xn}
X^{N,k, \alpha}_t
=~&
\xi_k
+ \int_0^t b \big(s, X^{N,k, \alpha}_{s \wedge \cdot}, \mu^{N, \alpha}_s, \alpha_s \big) \d s
+ \int_0^t \sigma \big(s, X^{N,k, \alpha}_{s \wedge \cdot}, \mu^{N, \alpha}_s, \alpha_s \big) \d W^k_s \nonumber \\
&~~~~~~~~~~~~~~~~~~~~~~~~~~~~~~+ \int_0^t \sigma_0 \big(s, X^{N,k, \alpha}_{s \wedge \cdot}, \mu^{N, \alpha}_s, \alpha_s \big) \d B_s,
~~k=1, \ldots, N,  \\
\mu^{N, \alpha}_t :=~&
\frac1N \sum_{k=1}^N \delta_{X^{N,k, \alpha}_{t \wedge \cdot}},\;t\ge 0.\nonumber
\end{align}
The corresponding centralized optimal control problem is   then defined as 
\begin{equation} \label{eq:def_Vn0}
V^N := \sup_{\alpha \in \Ac_N} J_N(\alpha),
~~\mbox{with}~
J_N(\alpha) := \E^{\P^N} \Big[ \int_0^T L (t, \mu^{N, \alpha}_t, \alpha_t ) \d t + g(\mu^{N, \alpha}_T) \Big].
\end{equation}
\begin{remark}
$\mathrm{(i)}$ Let us emphasize  that in \eqref{eq:def_Xn}-\eqref{eq:def_Vn0}
the controller has access to the full information $\F^N$, and that all agents  share the same control process $\alpha$.
The fact that controls are $\F^N$-adapted is important to make the problem \eqref{eq:def_Vn0} time consistent, so that the classical dynamic programming approach can be applied, in particular to provide a PDE characterization and derive numerical resolution schemes.
\vspace{0.5em}
 
\noindent $\mathrm{(ii)}$ Such a {centralized optimal control} problem
appears  very naturally in many social/economic problems, in which the same control applies to all agents, 
e.g. the tax rate, retirement age, etc.
 
\end{remark}

\section{Convergence and convergence rate in a general path-dependent setting}
\label{sec:path_depend}
 
In this section, we stay in a path-dependent setting and study the propagation of chaos by using both the weak convergence approach and BSDE techniques.
 
\subsection{The weak convergence approach}
\label{subsec:weak_cvg}
 
{In this section, we prove convergence of the finite population problems to the original mean-field problem by relying on the weak convergence approach.
It has already been used in the context of classical mean-field control problems without common noise by Lacker \cite{LackerLimit},
and by Djete, Possama\"i and Tan \cite{DjeteApprox} in the common noise setting.}
The main difference here is that  control processes in \eqref{eq:def_V} are required to be adapted to the common noise filtration $\G$ only (and not to $\F$).
 
\vspace{0.5em}
 
Let us impose some technical conditions.
Recall that $A$ is a non-empty Polish space, in which we fixed an element $a_0 \in A$. Hereafter, the maps
$$
(b, \sigma, \sigma_0) : [0,T] \x \Cc^d \x \Pc_2(\Cc^d \x A) \x A \longrightarrow \R^d \x \S^d \x \S^d,
$$
and
$$
L : [0,T] \x \Pc_2(\Cc^d) \longrightarrow \R,
~~~
g: \Pc_2(\Cc^d) \longrightarrow \R,
$$
are all Borel measurable.
\begin{assumption} \label{assum:main}
$\mathrm{(i)}$ The function $(b, \sigma, \sigma_0)$ are all non-anticipative in the sense that
$$
(b, \sigma, \sigma_0) (t, \xb, \nu, a) = (b, \sigma, \sigma_0) (t, \xb_{t \wedge \cdot}, \nu, a),
$$
for all $(t, \xb, \nu, a) \in [0,T] \x \Cc^d \x \Pc_2(\Cc^d) \x A$.
Moreover, there exists a constant $C> 0$ such that
$$
\big| (b, \sigma, \sigma_0) (t, \xb, \nu, a) - (b, \sigma, \sigma_0) (t, \xb', \nu', a) \big|
\le
C \big( \| \xb - \xb' \| + \Wc_2 (\nu, \nu') \big),
$$
and
$$
\big| (b, \sigma, \sigma_0)(t, \xb, \nu, a) \big|^2
\le
C \Big( 1 + \|\xb\|^2 + \int_{\Cc^d} \| \yb\|^2 \nu(\d \yb) + \rho(a, a_0)^2 \Big),
$$
for all $(t, \xb, \xb', \nu, \nu', a) \in [0,T] \x \Cc^d \x \Cc^d \x \Pc_2(\Cc^d) \x \Pc_2(\Cc^d) \x A$.
\vspace{0.5em}

\noindent $\mathrm{(ii)}$ The function $g$ is lower semi-continuous and 
 $(\nu,a)\in\Pc_2(\Cc^d)\x A \mapsto L(t, \nu, a)$ is lower semi-continuous, for each $t\in  [0,T]$.
Moreover, there are constants $C>0$, $p > 2$ and $C_p > 0$, such that
\begin{equation} \label{eq:L_growth}
- C \Big( 1 + \int_{\Cc^d} \| \yb\|^2 \nu(\d \yb) \Big)
~\le~
L(t, \nu, a)
~\le~
C \Big( 1 + \int_{\Cc^d} \| \yb\|^2 \nu(\d \yb) \Big)
- C_p~ \rho(a, a_0)^p,
\end{equation}
and
$$
\big| g(\nu) \big|
~\le~
C \Big( 1 + \int_{\Cc^d} \| \yb\|^2 \nu(\d \yb) \Big),
$$
for all $(t, \nu, a) \in [0,T] \x \Pc_2(\Cc^d) \x A$.
\end{assumption}
 
\begin{remark}
$\mathrm{(i)}$ The Lipschitz and growth conditions on $(b, \sigma, \sigma_0)$ {in Assumption \ref{assum:main}.$\mathrm{(i)}$} ensure that the SDEs \eqref{eq:def_X_alpha} and \eqref{eq:def_Xn} have unique solutions for every admissible control $\alpha$  (see e.g.~\cite[Theorem A.3.]{DjeteDPP}).
\vspace{0.5em}

\noindent $\mathrm{(ii)}$ The growth condition on $g$ and $L$
ensures that $J(\alpha)$ in \eqref{eq:def_V} (resp. $J_N(\alpha)$ in \eqref{eq:def_Vn0}) is well defined
for all admissible controls $\alpha \in \Ac_0$ (resp. $\alpha \in \Ac_N$).
On the other hand, the term $- C_p \rho(a, a_0)^p$ (with $p > 2$) in \eqref{eq:L_growth}
induces further that the value functions $V$ and $V^N$ are all finite. {It will also ensure} the  tightness  of   maximizing sequences in \eqref{eq:def_V} and  \eqref{eq:def_Vn0}.
\end{remark}

\paragraph{A weak formulation of the control problem.}
 
Following \cite{DjeteApprox}, let us define a weak control with initial condition $\nu_0 \in \Pc_2(\R^d)$ as a term
\begin{equation} \label{eq:weak_ctrl_gamma}
\gamma
=
\big(\Om^{\gamma}, \Fc^{\gamma}, \P^{\gamma}, \F^{\gamma} = (\Fc^{\gamma}_t)_{t \le T}, \G^{\gamma} = (\Gc^{\gamma}_t)_{t \le T}, X^{\gamma}, W^{\gamma}, B^{\gamma}, \mu^{\gamma}, \alpha^{\gamma}
\big),
\end{equation}
which satisfies the following conditions:
\begin{enumerate}
\item The triple $(\Om^{\gamma}, \Fc^{\gamma}, \P^{\gamma})$ is a probability space, equipped with two filtration $\F^{\gamma}$ and $\G^{\gamma}$ such that
$$
\Gc^{\gamma}_t \subseteq \Fc^{\gamma}_t
~\mbox{and}~
\E^{\P^{\gamma}}[\1_{D} | \Gc^{\gamma}_t ] = \E^{\P^{\gamma}} [\1_{D} | \Gc^{\gamma}_T], ~\P^{\gamma}\mbox{-a.s., for all}~
D \in \Fc^{\gamma}_t \vee \sigma(W^{\gamma}),
~\mbox{and}~
t \le T.
$$
\item The process $X^{\gamma}$ is $\R^d$-valued, continuous and $\F^{\gamma}$-adapted,
and $\alpha^{\gamma}$ is a $A$-valued $\G^{\gamma}$-predictable process such that
$$
\E^{\P^{\gamma}}
\Big[ \big\| X^{\gamma} \big\|^2 + \int_0^T  \rho \big(\alpha^{\gamma}_t, a_0 \big) ^2 \d t \Big]
< 
\infty.
$$
\item The couple $(W^{\gamma}, B^{\gamma})$ is a Brownian motion w.r.t. $\F^{\gamma}$,  $B^{\gamma}$ is adapted to $\G^{\gamma}$,
$\Fc^{\gamma}_0 \vee \sigma(W^{\gamma})$ is independent of $\Gc^{\gamma}_T$,
and $\mu^{\gamma} = (\mu^{\gamma}_t)_{t\le T}$ is a $\Pc_2(\C^d)$-valued $\G^{\gamma}$-predictable process such that
$$
\mu^{\gamma} = {\left(\Lc^{\P^{\gamma}}(X^{\gamma}_{t \wedge \cdot} | \Gc^{\gamma}_t )\right)_{t\le T}},
~\d \P^{\gamma} \x \d t \mbox{-a.s.}
$$
\item The process $X^{\gamma}$ satisfies $\P^{\gamma} \circ (X^{\gamma}_0)^{-1} = \nu_0$ and 
$$
X^{\gamma}_t = X^{\gamma}_0 + \int_0^t b(s, X^{\gamma}_{s \wedge \cdot}, \mu^{\gamma}_s, \alpha^{\gamma}_s) \d s
+\int_0^t \sigma(s, X^{\gamma}_{s \wedge \cdot}, \mu^{\gamma}_s, \alpha^{\gamma}_s) \d W^{\gamma}_s
+\int_0^t \sigma_0(s, X^{\gamma}_{s \wedge \cdot}, \mu^{\gamma}_s, \alpha^{\gamma}_s) \d B^{\gamma}_s,
$$
for $t \ge 0$, $\P^{\gamma}${-a.s.}
\end{enumerate}
Let  $\Gamma$ be the collection of all weak controls. 
Then, the weak formulation of our mean-field control problem is  given by
\begin{equation} \label{eq:def_VW}
V_W := \sup_{\gamma \in \Gamma} \E^{{ \P^{\gamma}}} \Big[ \int_0^T L(t, \mu^{\gamma}_t, \alpha^{\gamma}_t ) \d t + g(\mu^{\gamma}_T) \Big].
\end{equation}
\begin{remark} \label{rem: ega fort faible}
{Let $\alpha \in \Ac_0$ be an admissible control in the filtered probability space $(\Om, \Fc, \F, \P)$ equipped with the Brownian motion $W$ and $B$, and the common noise filtration $\G = (\Gc_t)_{t \le T}$.}
Then, with $(X^{\alpha}, \mu^{\alpha})$ being defined in \eqref{eq:def_X_alpha}, it is easy to check that
$$
\big(\Om, \Fc, \P, \F, \G, X^{\alpha}, W, B, \mu^{\alpha}, \alpha \big) \in \Gamma.
$$
This implies that a strong control $\alpha \in \Ac_0$ induces a weak control, so that $V \le V_W$.
Together with the semi-continuity conditions on $L$ and $g$ in Assumption \ref{assum:main}, we will further deduce that $V = V_W$.
\end{remark}
\begin{theorem} \label{thm:weak_cvg}
Let Assumption \ref{assum:main} hold true.
Then the strong formulation and weak formulations have the same value, i.e.
$$
V_W = V.
$$
Assume in addition that $\sigma_0$ is not controlled (i.e. $\sigma_0(\cdot, a)$ is independent of $a\in A$), 
{and that there exists $p > 2$ such that $\nu_k \in \Pc_p(\R^d)$ for all $k \ge 0$ and $\Wc_p \big(N^{-1}\sum_{k=1}^{N} \nu_k , \nu_0 \big) {\longrightarrow}0$ as $N\longrightarrow \infty$.}
Then,
$$
V^N
~\longrightarrow~
V
~\mbox{as}~
N \longrightarrow \infty.
$$
\end{theorem}
The proof is mainly adapted from  \cite{DjeteApprox}. 
It is slightly different because our controls have to be adapted to the common noise filtration $\G$.
\vspace{2mm}

In addition to the convergence of the value function $V^N$, one can also deduce the existence of an optimal relaxed control for  \eqref{eq:def_V},
as well as the weak convergence of the finite population problem to the mean-field one.
As the techniques are quite classical, these additional results and proofs are reported in Section \ref{sec:proof_weak_cvg} in Appendix.

\subsection{The Backward SDE approach}
 
While the weak convergence approach can be used to obtain a general convergence result,
it does not provide a convergence rate.
In the following, we consider the case where only the drift is controlled {and apply a BSDE approach},
from which we derive the  optimal convergence rate under  Lipschitz conditions on the reward function in the Wasserstein sense.

\paragraph{The set-up.}
	Given the filtered probability space  $(\Om, \Fc, \F, \P)$, let $S^2$ denote the space of all $\F$-adapted and continuous $\R$-valued processes $(Y_t)_{t\le T}$ such that $\E \big[ \|Y \|^2 \big] < \infty$,
and let $H^2$ denote the space of all $\F$-optional $\R^d$-valued processes $(Z_t)_{t\le T}$ such that $\E \big[ \int_0^T |Z_t|^2 dt \big] < \infty$.
\vspace{0.5em}
In the following, we consider a   setting with uncontrolled volatility coefficients. 
More precisely, we are equipped with  coefficients
$$
(b_0, \sigma_0, \sigma): [0,T] \x \Cc^d \x \Pc_2(\Cc^d) \longrightarrow \R^d \x \S^d \x \S^d,
~\mbox{and}~
b_1: [0,T] \x \Pc_2(\Cc^d) \x A \longrightarrow \R^d,
$$
such that
\begin{equation} \label{eq:def_BSDE_b}
b(s, \xb, \nu, a)
=
b_0(s, \xb, \nu)
+
\sigma_0 (s, \xb, \nu) b_1 \big(s, \nu, a \big),\;(s,\xb,\nu,a)\in [0,T]\x \Cc^d\x \Pc_2(\Cc^d)\x A. 
\end{equation}
\begin{assumption} \label{assum:drift_control}
$\mathrm{(i)}$ { $(b_0, b_1, b, \sigma_0, \sigma)$ are non-anticipative, and satisfy the same Lipschitz and growth conditions as $(b, \sigma_0, \sigma)$ in Assumption \ref{assum:main}.}
 
\vspace{0.5em}

\noindent $\mathrm{(ii)}$  $b_1: [0,T] \x \Pc_2(\Cc^d) \x A \longrightarrow \R^d$ is bounded and measurable,
and { the set $A$ is bounded in the sense that  $\sup_{a \in A} \rho(a, a_0) < \infty$.}

\vspace{0.5em}
 
\noindent $\mathrm{(iii)}$ There exists a constant $C>0$ such that,
for all $(t, \nu, a) \in [0,T] \x \Pc_2(\Cc^d) \x A$,
\begin{equation} \label{eq:L_growth}
\big| L(t, \nu, a) \big|^2
+
\big| g (\nu) \big|^2
~\le~
C \Big( 1 + \int_{\Cc^d} \| \yb\|^2 \nu(\d \yb) \Big).
\end{equation}
\end{assumption}
Then, we define $(X,\mu)$ by the uncontrolled SDE
\begin{equation} \label{eq:MKV_SDE}
X_t
=
\xi
+ \int_0^t \! b_0(s, X_{s \wedge \cdot}, \mu_s) \d s
+ \int_0^t \!\! \sigma(s, X_{s \wedge \cdot}, \mu_s) ~\d W_s
+ \int_0^t \!\! \sigma_0(s, X_{s \wedge \cdot}, \mu_s) ~\d B_s,
~t\le T,
\end{equation}
and
\begin{equation} \label{eq:def_mut}
\mu_t
=
\Lc^{\P}(X_{t \wedge \cdot} | \Gc_t)
=
\Lc^{\P}(X_{t \wedge \cdot} | \Gc_T)
=
\Lc^{\P}(X_{t \wedge \cdot} | B), ~t\le T,~ \mbox{a.s.}
\end{equation}
We refer to e.g.~\cite[Theorem A.3.]{DjeteDPP} for the existence and uniqueness of a solution to the McKean-Vlasov SDE \eqref{eq:MKV_SDE}-\eqref{eq:def_mut}.
 
\vspace{0.5em}

{Notice that  the admissible control set $\Ac_0$ now contains all $A$-valued $\G$-predictable processes, since   $A$ is assumed to be bounded  in Assumption \ref{assum:drift_control}.}

For each $\alpha \in \Ac_0$, let us define an equivalent probability measure $\Q^{\alpha}$ on $(\Om, \Fc)$ by
\begin{equation} \label{eq:def_Qa}
\frac{d \Q^{\alpha} }{d \P}
~=~
{\mathcal{E} \Big( \int_0^\cdot  b_1 \big(t, \mu_t, \alpha_t \big)~ \d B_t \Big)_T}
\end{equation}
{in which $\mathcal{E}(M )$ denotes  the Doléans-Dade's (local)  martingale associated to a (local) martingale $M$.}
It is clear that
\begin{align}\label{eq: def Balpha}
B^{\alpha}_t := B_t - \int_0^t b_1 \big(s, \mu_s, \alpha_s \big) \d s
~~\mbox{is a}~(\G, \Q^{\alpha}) \mbox{-Brownian motion},
\end{align}
and that $X$ satisfies
\begin{align} \label{eq:X_SDE}
X_t
=
\xi + \int_0^t b(s, X_{s \wedge \cdot}, \mu_s, \alpha_s) ds
&+ \int_0^t \sigma(s, X_{s \wedge \cdot}, \mu_s) dW_s \nonumber \\
&~+ \int_0^t \sigma_0(s, X_{s \wedge \cdot}, \mu_s) dB^{\alpha}_s,
~~t\le T, ~\Q^{\alpha}\mbox{-a.s.}
\end{align}
The following results shows that $\mu_t$ is also the time $t$ conditional distribution of $X$ under $\Q^{\alpha}$.
 
\begin{lemma} \label{lemm:identic_mu}
Let Assumption \ref{assum:drift_control} hold.
Then for every control process $\alpha \in \Ac_0$,
$$
\mu_t
= \Lc^{\Q^{\alpha}}(X_{t \wedge \cdot} | \Gc_t)
= \Lc^{\Q^{\alpha}}(X_{t \wedge \cdot} | \Gc_T),
~t\le T,~ \Q^{\alpha} \mbox{-a.s.}
$$
\end{lemma}
\begin{proof}
By \cite[Theorem A.3.]{DjeteDPP}, the solution to \eqref{eq:MKV_SDE}-\eqref{eq:def_mut} is unique,
and $ \mu$ is adapted to the filtration generated by $B$.
Further, by {\cite[Theorem 1.5]{Kurtz},} there exists a Borel measurable function $\Psi: \R^d \x \Cc^d \x \Cc^d \longrightarrow \Cc^d$
such that
$$
X = \Psi(\xi, W, B),
~ \mbox{a.s.}
$$
Since the density $\d \Q^{\alpha} / \d \P \in \Gc_T$ and $(\xi, W)$ is independent of $\Gc_T$ under $\P$,
it is clear that
$$
\Lc^{\P} \big((\xi, W) | \Gc_T \big)
~=~
\Lc^{\Q^{\alpha}} \big( (\xi, W) | \Gc_T \big),
~\mbox{a.s.}
$$
This is enough to conclude the proof.
\end{proof}

\paragraph{Representation of the mean-field control problem by a BSDE.}
 
Let us consider the mean-field control problem: 
\begin{equation} \label{eq:BSDE_value_0}
\Vb_0
~:=~
\sup_{\alpha \in \Ac_0}
\E^{\Q^{\alpha}} \Big[ \int_0^T L(t, \mu_t, \alpha_t) dt + g( \mu_T) \Big].
\end{equation}

\begin{remark}
\label{rem:BSDE_ctrl}
When $\nu_0 \in \Pc_p(\R^d)$ with $p > 2$, $(b, b_1, \sigma, \sigma_0, L, g)$ satisfy the same continuity and growth conditions as $(b, \sigma, \sigma_0, L ,g)$ in Assumption \ref{assum:main},
one can deduce that  \eqref{eq:BSDE_value_0} is equivalent to the strong and weak formulation   in \eqref{eq:def_V} and \eqref{eq:def_VW}.
Indeed, given a control $\alpha \in \Ac_0$, the term
$(\Om, \Fc, \Q^{\alpha}, \F, \G, X, \mu, \alpha)$ is a weak control term in the sense of \eqref{eq:weak_ctrl_gamma}, so that $ \Vb_0\le V_W = V$, where the last equality follows from Theorem \ref{thm:weak_cvg}.
Next, as shown in the proof of  Theorem \ref{thm:weak_cvg}  below (see in particular Proposition \ref{prop:density_Pcb}), any strong control  $\alpha \in \Ac_0$ can be approximated by piecewise constant strong controls $(\alpha^n)_{n \ge 1} \in \Ac_0^{\circ}$ (see \eqref{eq:PcbS0} for a precise definition).
For any such piecewise constant strong control $\alpha^n$, one can then easily construct a control $\bar \alpha^n \in \Ac_0$ such that
$$
\Lc^{\P} \big(B, \alpha^n, X^{\alpha^n}, \mu^{\alpha^n} \big)
~=~
\Lc^{\Q^{\bar \alpha^n}} \big( B^{\bar \alpha^n}, \bar \alpha^n, X, \mu \big).
$$
This is due to the fact that any piecewise strong control $\alpha^n$ can be viewed as a function of the stochastic process $\int_0^{\cdot} b_1(s, \mu_s, \alpha^n_s) \d s + B_{\cdot} $, say  $\alpha^n_{\cdot}=\alpha^n_{\cdot} \big(\int_0^{\cdot} b_1(s, \mu_s, \alpha^n_s) \d s + B_{\cdot} \big) $, to which we can associate $\bar \alpha^n_{\cdot} := \alpha^n_{\cdot} \big(\int_0^{\cdot} b_1(s, \mu_s, \bar  \alpha^n_s) \d s + B^{\bar \alpha^n}_{\cdot} \big) $.
This implies that $ \Vb_0 \ge V$. Hence,
$$
	\Vb_0 = V_W = V.
$$
\end{remark}

Next, let us define the Hamiltonian $H: [0,T] \x \Pc_2(\Cc^d) \x \R^d \longrightarrow \R$ by
\begin{equation} \label{eq:defH}
H(\cdot,z)  ~:=~ \sup_{a \in A} \Big( L(\cdot, a) + b_1 (\cdot, a) \cdot z \Big),\;z\in \R^d.
\end{equation}
 
\begin{theorem}
Let Assumption \ref{assum:drift_control} hold true.
Then, the following BSDE
\begin{equation} \label{eq:BSDE}
Y_t = g(\mu_T) - \int_t^T H(s, \mu_s, Z_s) ds - \int_t^T Z_s \cdot dB_s, ~~t\le T, ~\P\mbox{-a.s.}
\end{equation}
has a unique solution $(Y, Z) \in S^2 \x H^2$, and
$$
Y_0 = \Vb_0,
~\mbox{a.s.}
$$
Moreover, a control process $\alpha^* \in \Ac_0$ is   optimal   for \eqref{eq:BSDE_value_0} if and only if
\begin{equation} \label{eq:cond_equality_H}
H(\cdot, \mu, Z)
=
L(\cdot, \mu, \alpha^*) + b_{{1}} (\cdot, \mu, \alpha^*) \cdot Z,
~\d \P \x \d t ~\mbox{-a.e.  on $[0,T]$.}
\end{equation}
\end{theorem}
\begin{proof}
First,   notice that $z \longmapsto H(\cdot, z)$ is globally Lipschitz as $b_1$ is uniformly bounded.
Moreover,   the growth conditions on $g$ and $L$ in \eqref{eq:L_growth} imply that 
$$
\E^{\P} \Big[ \big| g( \mu_T) \big|^2 + \int_0^T \big| H(t, \mu_t, 0 ) \big|^2 dt \Big] < \infty.
$$
It follows that   \eqref{eq:BSDE} has a unique solution in $ S^2 \x H^2$.
 
\vspace{0.5em}
Next, let us define  $\Vb = (\Vb_t)_{t\le T}$ by
$$
\Vb_t ~:=~ \mathrm{ess\!\!}\sup_{\alpha \in \Ac_0}
\E^{\Q^{\alpha}} \Big[
\int_t^T L(s, \mu_s, \alpha_s) \d s + g( \mu_T)
~\Big| \Gc_t \Big].
$$
Then, by the dynamic programming principle, see e.g.~{\cite[Theorem 3.2]{DjeteDPP},}
$$
\Vb_t ~:=~ \mathrm{ess\!\!}\sup_{\alpha \in \Ac_0}
\E^{\Q^{\alpha}} \Big[
\int_t^{\tau} L(s, \mu_s, \alpha_s) \d s + \Vb_{\tau}
~\Big| \Gc_t \Big]
$$
for any time $t\le T$ and $\G$-stopping time $\tau$ taking values in $[t,T]$.
Moreover, as $b_1$ is uniformly bounded, and $g$ and $L$ have linear growth in $(\xb, \nu)$ (see \eqref{eq:L_growth}),
$\Vb$ is a square integrable process under any $\Q^{\alpha}$.
Thus, for any $\alpha \in \Ac_0$, the process
$$
\Big( \Vb_t + \int_0^t L(s, \mu_s, \alpha_s) ds \Big)_{t\le T}
~\mbox{is a}~{(\G,\Q^{\alpha})} \mbox{-supermartingale}.
$$
It thus follows from the Doob-Meyer's decomposition and   the martingale representation theorem that
there exists a process $\overline Z$ such that
$$
\Vb_t
\le
g(\mu_T) + \int_t^T \!\!\! \big( L(s, \mu_s, \alpha_s) + b_{{1}}(s, \mu_s, \alpha_s) \cdot \overline Z_s \big) ds
- \int_t^T\!\!\! \overline Z_s \cdot \d B_s,
~t \in [0,T],
~\mbox{for all}~
\alpha \in \Ac_0.
$$
Recalling the definition of $H$ in \eqref{eq:defH},   
it follows from the comparison principle  for BSDEs that { $ \Vb_t \le Y_t$ for all $t \in [0,T]$.}
 
\vspace{0.5em}
Next, for any $\eps > 0$, one can select {an admissible} control process $\alpha^{\eps}$ such that
\begin{equation} \label{eq:BSDE_comparison}
L (\cdot, \mu , \alpha^{\eps}) + b_{{1}} (\cdot, \mu , \alpha^{\eps}) \cdot Z 
~\ge~
H( \cdot, \mu , Z ) - \eps,
~~\d \P \x \d t \mbox{-a.e.}
\end{equation}
In view of \eqref{eq: def Balpha}, it follows that 
{
$$
	Y_t ~\le~ \E^{\Q^{\alpha^{\eps}}} \Big[ \int_t^T L(s, \mu_s, \alpha^{\eps}_s) \d s + g(\mu_T) ~\Big| \Gc_t \Big] + \eps (T-t) 
	~\le~ \Vb_t + \eps (T-t), ~\mbox{a.s.}
$$
Therefore,  $ \Vb_t = Y_t$, a.s., for all $t \in [0,T]$.}
Moreover, the above argument shows that a control $\alpha^*$ is an optimal control if and only if the inequality in \eqref{eq:BSDE_comparison} becomes an equality, or equivalently, \eqref{eq:cond_equality_H} holds true.
\end{proof}

\begin{remark}
Let us consider the Markovian setting in which 
the dependance of the coefficient  $\big( b_0, b, \sigma_0, \sigma, L, g \big) (t, \xb, \nu_t, \cdot)$ on $\xb$ is only through $\xb_t$
and their dependance on $\nu_t = \Lc(X_{t \wedge \cdot}|B)$ is only through $\nu^0_t := \Lc(X_t|B)$  at time $t$.
Then, it is easy to construct a measurable map $U: [0,T] \x \Pc_2(\R^d) \longrightarrow \R$ such that
$$
Y_t = U(t, \nu^0_t),
~~\mbox{with}~\nu^0_t := \Lc^{\P}(X_t|B),
~~\mbox{for}~t\le T,~\P\mbox{-a.s.}
$$
Assume that $U \in C^{0, 1}$,
then it follows from the $C^1$-It\^o's formula {\cite[Theorem 2.3]{BouchardTanWang}} that
$$
(Z_t)_{t\le T}
=
\big(\big \langle \nu^0_t, { D_m U(t, \nu^0_t, \cdot) \sigma_0(t, \cdot, \nu^0_t)} \big \rangle \big)_{t\le T},
~\d t \x \d \P\mbox{-a.e.},
$$
in which the derivative operator $D_m$ is defined as in \eqref{eq:def_DmF} below.
\end{remark}

\paragraph{The propagation of chaos.}
 
To study the propagation of chaos, we   define the controlled particle system on the same filtered probability space 
$(\Om, \Fc, \F, \P)$ equipped with $(\xi, W, B)$. First, we enlarge it in order to include in addition a sequence of $(\xi_k, W^k)_{k \ge 1}$ such that
$\xi_k \sim \nu_k \in \Pc_2(\R^d)$, $\xi_k \in \Fc_0$ and $W^k$ is a $\F$-standard Brownian motion, for each $k\ge 1$.
By abuse of notation, the enlarged space is still denoted by $(\Om, \Fc, \F, \P)$.
Moreover, $(\xi, W, B)$ and $(\xi_k, W^k, B^k)_{k \ge 1}$ are mutually independent.

We define $(X^{N,k})_{k\le N}$ by the (uncontrolled) system: 
\begin{align*}
X^{N,k}_t
=~&
X^k_0
+ \int_0^t \! b_0 (s , X^{N,k}_{s\wedge \cdot}, \mu^N_s) \d s
+ \int_0^t \!\! \sigma(s , X^{N,k}_{s\wedge \cdot}, \mu^N_s) \d W^k_s
+ \int_0^t \!\! \sigma_0(s, X^{N,k}_{s\wedge \cdot}, \mu^N_s) \d B_s, \\
\mbox{ with }\mu^N_t
:=~& \frac1N \sum_{k=1}^N \delta_{X^{N,k}_{t \wedge \cdot}},\; t\le T,\; N\ge 1.
\end{align*}
Next,  we assume that $\sigma_0$ is non-degenerate and define,  for each $N \ge 1$, $(Y^N, {Z^{N,1},\ldots,Z^{N,N}})$ as the solution to the BSDE
\begin{align} \label{eq:BSDE_YN}
Y^N_t &= g( \mu^N_T)
- \!\int_t^T \!\! H(s, \mu^N_s, Z^N_s) \d s \\
&~~~~~ - \sum_{{k=1}}^N \int_t^T \!\!
\big( { \sigma_0^{-1}\sigma }\big) (s, X^{N,k}_{s \wedge \cdot}, \mu^{{N}}_s) Z^{N,k}_s \cdot
\d W^k_s
- \int_t^T Z^N_s \cdot \d B_s, \;t\le T,\nonumber\\
\mbox{ with } Z^N &:= \sum_{k=1}^N Z^{N,k}.
\end{align}
 To see that  existence and uniqueness of a solution  $(Y^N, {Z^{N,1},\ldots,Z^{N,N}}) \in \S^2\x (\H^2)^N$ to the above BSDE hold under Assumption \ref{assum:drift_control}, it suffices to refer to the classical BSDEs' theory using the martingale representation applied to 
$(  \int_0^\cdot\big( \sigma_0^{-1} \sigma \big) \big(s, X^{N,k}_{s \wedge \cdot}, \mu^{N}_s \big) \d W^k_s + B)_{k = 1, \ldots, N}$.
By standard arguments, $Y^N$ provides the value process of the following optimal control problem:
\begin{equation} \label{eq:def_VbN}
\Vb^N_0
:=
\sup_{\alpha \in \Ac_N}
\E^{\Q^{N,\alpha}} \Big[
\int_0^T L \big(t, \mu^{N, \alpha}_t, \alpha_t \big) dt + g \big( \mu^{N, \alpha}_T \big)
\Big],
\end{equation}
where, by a slight abuse of  the notations in Section \ref{subsec:NctrlPb}, we denote by $\Ac_N$ the space of all $A$-valued $\F^N$-predictable processes with $\F^N$ being the filtration generated by $(B, W^1, \ldots, W^N)$, and for each $\alpha \in \Ac_N$,  $\Q^{N, \alpha}$ is the probability measure equivalent to $\P$, such that $\Wb = (\Wb^1, \ldots, \Wb^N)$ is a $\Q^{N, \alpha}$-Brownian motion, with
$$
\Wb^k:= \int_0^\cdot b_1 \big(X^{N,k}_{s \wedge \cdot}, \mu^N_s \big) \d s + \int_0^\cdot  \big( \sigma_0^{-1} \sigma \big) \big(s, X^{N,k}_{s \wedge \cdot}, \mu^{N}_s \big) \d W^k_s + B, ~~k = 1, \ldots, N.
$$
In particular, with $b$ given as in \eqref{eq:def_BSDE_b}, one has
\begin{align*}
X^{N,k}
=~&
\xi_k
+ \int_0^\cdot  b \big(s, X^{N,k}_{s \wedge \cdot}, \mu^{N}_s, \alpha_s \big) \d s
+ \int_0^\cdot \sigma_0 \big(s, X^{N,k}_{s \wedge \cdot}, \mu^{N}_s \big) \d \Wb^k_s,
~~k=1, \ldots, N.
\end{align*}
As in Remark \ref{rem:BSDE_ctrl}, {under  suitable semi-continuity conditions on $L$ and $g$}, it is easy to show that the above control problem $\Vb^N_0$ is equivalent to its strong formulation ${V^N}$ in \eqref{eq:def_Vn0}. {Then, the following result, based on standard BSDE technics, provides a convergence rate for $V^N$ towards $V$. }

\begin{theorem} \label{thm:BSDE_rate}
Let Assumption \ref{assum:drift_control} hold true { and suppose that $\sigma_0$ is non-degenerate at each point of $[0,T] \x \Cc^d \x \Pc_2(\Cc^d)$.} 
Then, for each $N \ge 1$, one has {$\Vb_0^N = Y^N_0$ and $\alpha^{{N}*}\in \Ac_N$ is optimal for  \eqref{eq:def_VbN}}
if and only if
$$
H(\cdot, \mu^N, Z^N)
=
L(\cdot, \mu^N, \alpha^{{N}*}) + b_{{1}}(\cdot, \mu^N, \alpha^{{N}*}) \cdot Z^N,
~\d \P \x \d t ~\mbox{-a.e. }\mbox{ on $[0,T]$.}
$$
Assume in addition that there exists a constant $C_0 > 0$ such that,
for all $(t, \xb, \nu_1, \nu_2, z) \in [0,T] \x \Cc^d \x \Pc_2(\Cc^d) \x \Pc_2(\Cc^d) \x \R^d$,
$$
\big| H(t, \nu_1, z) - H(t, \nu_2, z) \big|
+
\big| g(\nu_1) - g(\nu_2) \big|
~\le~
C_0 \Wc_2(\nu_1, \nu_2).
$$
Then, there exists a constant $C> 0$, independent of $N$, such that
\begin{align*}
&\E \Big[\|Y^N-Y\|^2 \Big]
+
\sum_{k=1}^N \E \Big[ \int_t^T \big| (\sigma_0^{-1} \sigma) (s, X^{N,k}_{s \wedge \cdot}, \mu^N_s) Z^{N,k}_s \big|^{2}ds \Big]
+
\E \Big[ \int_t^T \big| Z^N_s-Z_{s} \big|^{2}ds \Big] \\
\le&~ C ~ \E \Big[ \Wc_2 \big( \mu^N_T, \mu_T \big)^2 \Big].
\end{align*}
\end{theorem}
\begin{proof}
The equality {$\Vb^N_0=Y^N_0$} is standard,  see e.g.~\cite[Proposition 3.2]{EKPQ}. Next, the estimate follows from   the stability  of  BSDEs (see e.g.~\cite[Proposition 2.1]{EKPQ}) and the Lipschitz continuity of the coefficients, observing that $t\mapsto \Wc_2 \big( \mu^N_t, \mu_t \big)$ is non-decreasing.
\end{proof}

\begin{remark}
$\mathrm{(i)}$
The main results in Theorem \ref{thm:BSDE_rate} shows that the convergence rate for the optimal control problem is the same as in  the linear case, 
$\E \Big[ \Wc_2 \big( \mu^N_T, \mu_T \big)^2 \Big]$,
since $\mu$ and $\mu^N$ are all induced by uncontrolled SDEs.
In this sense, it is sharp.
\vspace{0.5em}

\noindent $\mathrm{(ii)}$
To estimate the convergence rate of $\E \Big[ \Wc_2 \big( \mu^N_{{T}}, \mu_{{T}} \big)^2 \Big]$,
a classical way consists in introducing a sequence of processes $(\Xb^k)_{k \ge 1}$, where each $\Xb^k$ is defined as in \eqref{eq:MKV_SDE} but driven by $(W^k, B)$ and  with initial condition $\xi_k $, i.e.
$$
\Xb^k_t
=
\xi_k
+ \int_0^t \! b_0(s, \Xb^k_{s \wedge \cdot}, \mu_s) \d s
+ \int_0^t \!\! \sigma(s, \Xb^k_{s \wedge \cdot}, \mu_s) ~\d W^k_s
+ \int_0^t \!\! \sigma_0(s, \Xb^k_{s \wedge \cdot}, \mu_s) ~\d B_s,
$$
with $\mu_s = \Lc(\Xb^k_{s \wedge \cdot} | B) = \Lc(X_{s \wedge \cdot} |B)$.
We next define $\mub_t^{{N}} := \frac1N \sum_{k=1}^N \delta_{\Xb^k_{t \wedge \cdot}}$, $t\le T$, and consider the inequality
\begin{equation} \label{ineq:mun_mu}
\E \Big[ \Wc_2 \big( \mu^N_T, \mu_T \big)^2 \Big]
~\le~
2 \E \Big[ \Wc_2 \big( \mu^N_T, \mub^N_T \big)^2 \Big] + 2 \E \Big[ \Wc_2 \big( \mub^N_T, \mu_T \big)^2 \Big].
\end{equation}
When $(\xi_k)_{k \ge 1}$ is i.i.d. and follows the same distribution as $\xi$, i.e. $\nu_k = \nu_0$ for all $k \ge 1$,
it follows from standard analysis (see e.g.~\cite{McKean} or  \cite{Sznitman})
that the first term satisfies
$$
\E \Big[ \Wc_2 \big( \mu^N_T, \mub^N_T \big)^2 \Big]
~\le~
\frac{C}{N},\;N\ge 1,
~\mbox{for some constant}~
C> 0.
$$
The second term on the r.h.s.~of \eqref{ineq:mun_mu} is the distance between the (conditional) law $\mu_T$ and the empirical law induced by (conditional) i.i.d. random elements as in the Law of Large Number.
In the Markovian setting (i.e.~when $\mu_t$ just denotes the  marginal distribution of $X_t$), and when $\mu_t$ has finite $p$-moment for some $p > 2$,
  \cite{FournierG} provides   sharp convergence rates.
\end{remark}

\begin{remark}
When $A \subset \R^d$ and $b_{{1}}( \cdot, a) \equiv a$, it is well-known that the sub-gradient term $D_z H(t, \mu_t, Z_t)$ (resp. $D_z H(t, \mu^{{N}}_t, Z^{{N}}_t)$ ) provides the optimal control associated to $\Vb_0$ in \eqref{eq:BSDE_value_0} (resp. $\Vb^N_0$ in \eqref{eq:def_VbN}).
Therefore, under continuity conditions on $D_z H$,
the convergence of $Z^N \longrightarrow Z$ implies the convergence of the optimal controls $\alpha^{N,*}\to {\alpha^*}$.
\end{remark}

\section{Regularity and convergence rate in the Markovian setting}
\label{sec:Markovian}

In the following, we   stay in a Markovian setting and study the regularity of the value function of our control problem, {in a simplified model}.
Let us first recall the notion of derivatives for functional of probability measures.
\begin{definition} \cite{Cardaliaguet} or \cite[Definition 5.43]{CarmonaDelarue}\label{functional linear derivative}
A function $F: \Pc_2(\R^d) \longrightarrow \R$ is said to have a linear functional derivative if there exists a function,
$$
\frac{\delta F}{\delta m}: \Pc_2(\R^d) \times \R^d \ni(\nu, x) \longmapsto \frac{\delta F}{\delta m}(\nu)(x) \in \R,
$$
{that is} continuous for the product topology, {such that,} for any bounded subset $\Kc \subset \Pc_2(\R^d)$, the function $\R^d \ni x \longmapsto$ $[\delta F / \delta m](\nu)(x)$ is at most of quadratic {growth uniformly} in {$\nu \in \Kc$}, and such that, for all $\nu$ and $\nu^{\prime}$ in $\Pc_2(\R^d)$, we have 
\begin{equation} \label{eq:def_dmF}
F(\nu^{\prime})-F(\nu)=\int_0^1 \int_{\R^d} \frac{\delta F}{\delta m}\left(t \nu^{\prime}+(1-t) \nu\right)(x) \left[\nu^{\prime}-\nu\right](\d x) \d t .
\end{equation}
Assuming further that, for any $\nu \in \Pc_2(\R^d)$, the function $\R^d \ni x \longmapsto [\delta F / \delta m]F(\nu)(x)$ is differentiable,
we define
\begin{equation} \label{eq:def_DmF}
D_m F(\nu, x):= \frac{\partial}{\partial x} \delta_m F (\nu, x),
~~\mbox{with}~~
\delta_m F(\nu, x) := \frac{\delta F}{\delta m}(\nu)(x), \; (x,\nu)\in \R^d\x \Pc_2(\R^d).
\end{equation}
Upon existence, we further define the second order derivative $D^2_{mm} F(\nu, x, \cdot)$ as the derivative of $\nu \in \Pc_2(\R^d) \longmapsto D_m F(\nu, x)$,  for all $(x,\nu)\in \R^d\x \Pc_2(\R^d)$,
and then define higher order derivatives similarly.
\end{definition}
We now assume that the coefficient   are all Markovian.
Let $t\le T$, $\nu \in \Pc(\Cc^d)$, and $X$ denote the canonical process on $\Cc^d$. We now let $\nu_t := \nu \circ X_t^{-1}$ be the marginal distribution of $X_t$ under $\nu$.
\begin{assumption} \label{assum:convexity}
$\mathrm{(i)}$ The function $\sigma_0$ {is identically equal to} a non-degenerate constant matrix $\sigma_0^{\circ} \in \S^d$,
and the coefficient functions $b_0$, $\sigma$, $L$ and $g$ are Markovian in the sense that,
for all $(t, \xb, \mu, a) \in [0,T] \x \Cc^d \x \Pc_2(\Cc^d) \x A$,
$$
\big(b_0, \sigma, L, g \big) (t, \xb, \nu, a) = \big(b^{\circ}_0, \sigma^{\circ}, L^{\circ}, g^{\circ} \big) \big(t, \xb_t, \nu_t, a \big),
$$
for some $(b^{\circ}_0, \sigma^{\circ}, L^{\circ}_0) : [0,T] \x \R^d \x \Pc_2(\R^d) \x A \rightarrow \R^d \x \S^d \x \R$ and $g^{\circ}: \Pc_2(\R^d) \rightarrow \R$.
 
\vspace{0.5em}
 
\noindent $\mathrm{(ii)}$
The set $A \subset \R^d$ is compact and convex,  $b_1 (\cdot, a) \equiv a$ for all $a\in A$, so that
$$
b (\cdot, a) = b^{\circ}_0  + \sigma^{\circ}_0 a,
~\mbox{for all}~
 a\in  A.
$$
Moreover,
$$
\mbox{the function}~
a \in A \longmapsto L^{\circ}(t, \mu, a)
~\mbox{is strictly concave for all}~
(t, \mu) \in [0,T] \x \Pc_2(\R^d).
$$

\noindent $\mathrm{(iii)}$ The coefficient $(b^{\circ}, \sigma^{\circ})$ are bounded and Lipschitz continuous,
and the derivatives $(\partial_x b^{\circ}, \partial_x \sigma^{\circ}) : [0,T] \x \R^d \x \Pc_2(\R^d) \longrightarrow \R^d \x \S^d$,
and $(D_m b^{\circ}, D_m \sigma^{\circ}) : [0,T] \x \R^d \x \Pc_2(\R^d) \x \R^d \longrightarrow \R^d \x \S^d$
exist, are bounded and Lipschitz continuous.
\vspace{0.5em}
 
\noindent $\mathrm{(iv)}$ The reward functions $(L^{\circ}, g^{\circ})$ and their derivatives $D_m L^{\circ}: [0,T] \x \Pc_2(\R^d) \x A \x \R^d \longrightarrow \R$ and $D_m g^{\circ}: \Pc_2(\R^d) \x \R^d \longrightarrow \R$ are all { bounded and uniformly continuous.}
\end{assumption}
 
For simplification, we abuse of notations and simply write 
$$
(b, \sigma, \sigma_0, L, g)
~\mbox{for}~
(b^{\circ}, \sigma^{\circ}, \sigma^{\circ}_0, L^{\circ}, g^{\circ}).
$$

Recall that $(\Om, \Fc, \F, \P)$ is a filtered probability space equipped with the Brownian motion $(W, B)$ and the random variable $\xi$, and that $\G$ is the filtration generated by $B$.
When $b_1(\cdot, a) \equiv a$ and $A \subset \R^d$ is a compact set,
$\Ac$ is in fact the space of all $A$-valued $\G$-progressively measurable processes.
 
Recall also that, given $\alpha \in \Ac$, the probability measure $\Q^{\alpha}$ is defined by \eqref{eq:def_Qa},
so that
$$
B^{\alpha}_t = B_t - \int_0^t \alpha_s \d s
~\mbox{is a}~ (\G, \Q^{\alpha}) \mbox{-Brownian motion}.
$$
Let
$$
\Qc
~:=~
\big\{
\Q^{\alpha} ~: \alpha \in \Ac
\big\}.
$$
For any $t\le T$ and random variable $\xi \in \Fc_0$ s.t. $\xi \sim \nu \in \Pc_2(\R^d)$, we denote by $X^{t,\xi}$ the solution of the McKean-Vlasov SDE
\begin{equation} \label{eq:SDE_constant_sigma0}
X_s
=
\xi
+ \int_t^s \! b_0(r, X_r, \mu^{t,\nu}_r) \d r
+ \int_t^s \!\! \sigma(r, X_r, \mu^{t,\nu}_r) ~\d W_r
+ \sigma_0 \big( B_s - B_t \big),
~s \in [t,T],
\end{equation}
with $\mu^{t,\nu}_s = \Lc(X^{t,\xi}_s | B)$ for all $s \in [t,T]$ { and $\mu^{t,\nu}_s \equiv \nu$ for $s \in [0,t]$}.
Then, by Lemma \ref{lemm:identic_mu}, one has
$$
\mu^{t,\nu}_s = \Lc^{\P} (X^{t, \xi}_s|B) = \Lc^{\Q} (X^{t,\xi}_s | B),
~~s \in [t,T],
~~\mbox{for any}~
\Q \in \Qc.
$$
Morever, under Assumption \ref{assum:convexity}, for any $\Q \in \Qc$, there exists a unique $\F^B$-progressively measurable process $\alpha^{\Q}$ together with
a $\Q$-Brownian motion $B^{\Q}$ such that
$$
B_t = B^{\Q}_t + \int_0^t \alpha^{\Q}_s ds, ~t\le T, ~\Q \mbox{-a.s.}
$$
Then, our McKean-Vlasov optimal control problem can be equivalently written as
\begin{equation} \label{eq:VJQ}
U(t, \nu)
~:=~
\sup_{\Q \in \Qc} J(t, \nu, \Q),
~~\mbox{where}~
J(t, \nu, \Q) := \E^{\Q} \Big[ \int_t^T L(s, \mu^{t, \nu}_s, \alpha^{\Q}_s) ds + g(\mu^{t, \nu}_T) \Big].
\end{equation}
 
\begin{remark}\label{rem : HJB eq}
$\mathrm{(i)}$ In the above Markovian setting, the master equation associated to the control problem \eqref{eq:def_V} is  
\begin{align}
\partial_t U (t, \nu)
& +
\sup_{a \in A} \Big(
L(t, \nu, a)
+
\int_{\R^d} b(t, x, \nu, a) \cdot D_m U (t, \nu, x) ~\nu(\d x)
\Big)\nonumber\\
&+
\frac12 \int_{\R^d} \mathrm{Tr} \big[ \big( \sigma \sigma^{\top} + \sigma_0 \sigma_0^{\top} \big) (t, x, \nu) ~\partial_x D_m U (t, \nu, x) \big]~ \nu(\d x)\nonumber \\
&+
\frac12 \int_{\R^d \x \R^d} \mathrm{Tr} \big[ \sigma_0  \sigma^{\top}_0  ~D^2_{mm} U (t, \nu, x, x') \big] ~\nu\otimes \nu(\d x, \d x')
= 0, \label{eq : Master eq U}
\end{align}
with terminal condition $U(T, \nu) = g(\nu)$.

\vspace{0.5em}
\noindent $\mathrm{(ii)}$ Similarly, for the centralized control problem \eqref{eq:def_Vn0}, the value function can be considered as a function of $(t, x_1, \ldots, x_N) \in [0,T] \x (\R^d)^N$,
which is characterized by the HJB equation
\begin{align} \label{eq:PDE_VN}
\partial_t & V^N (t, {x})
+
\sup_{a \in A} \Big(
{ L \big(t, m^{N,x}, a \big)}
+
\sum_{k=1}^N b(t, x_k, \cdot, a) \cdot D_{x_k} V^N (t, {x})
\Big) \nonumber \\
&+
\frac12 \sum_{k=1}^N \mathrm{Tr} \big( \sigma \sigma^{\top}(t, x_k, \cdot) D^2_{x_k x_k} V^N (t, {x}) \big)
+
\frac12 \sum_{k, \ell =1}^N \mathrm{Tr} \big( \sigma_0  \sigma^{\top}_0  D^2_{x_k, x_{\ell}} V^N (t, {x}) \big)
= 0,
\end{align}
with  terminal condition
\begin{equation} \label{eq:PDE_VN_term}
V^N (T, x ) = g \big( { m^{N,x} }\big),
\end{equation}
 where $m^{N,x} := \frac1N \sum_{k=1}^N \delta_{x_k}$ {and $x=(x_1,\ldots,x_N)$.}
\end{remark}

\subsection{$C^1$ regularity of the value function}

In this section, we focus on the $C^{0,1}$-regularity of the value fonction under the above conditions.

\begin{theorem} \label{thm:regularity}
Let Assumption \ref{assum:convexity} hold true.
Then, for each initial condition $(t, \nu) \in [0,T) \x \Pc_2(\R^d)$, there exists a unique maximizer for the optimization problem \eqref{eq:VJQ}.
Moreover,  $U \in C^{0,1}([0,T) \x \Pc_2(\R^d))$,
i.e. both $U: [0,T) \x \Pc_2(\R^d) \longrightarrow \R$ and $D_m U: [0,T] \x \Pc_2(\R^d) \x \R^d \longrightarrow \R$ are continuous.
\end{theorem}
 
Before turning to the proof of Theorem \ref{thm:regularity}, let us first provide some technical lemmas.
By its definition in \eqref{eq:SDE_constant_sigma0}, $\mu^{t,\nu}_s$ can be considered as a function of $(t,\nu)$ and the path of $B$.
In other words, we consider it as a function {$(t, \nu, \om^B) \in [0,T] \x \Pc_2(\R^d) \x \Cc^d \mapsto \mu^{t, \nu}_s (\om^B)$}.
\begin{lemma} \label{lemm:mu_Lip}
There exists a constant $C > 0$ such that
$$
\Wc_2 \big( \mu^{t_1, \nu_1}_s(\om^B_1), \mu^{t_2, \nu_2}_s(\om^B_2) \big)
~\le~
C
\Big( |t_1 - t_2| + \Wc_{{2}}(\nu_1, \nu_2) + \| \om^B_1 - \om^B_2 \| \Big),
$$
for all $(s, t_1, t_2, \nu_1, \nu_2, \om^B_1, \om^B_2) \in [0,T]^3\times (\Pc_2(\R^d))^2\times (\Cc^d)^2$.
\end{lemma}
\begin{proof}
Since {${b_0}$ and $\sigma$} are bounded and Lipschitz continuous, the above results follows immediately from the stability  of  \eqref{eq:SDE_constant_sigma0} (see e.g. \cite[Lemma 3.1]{BuckdahnLiPeng} in the case without common noise, whose extensions to the common noise setting is trivial).
\end{proof}

\begin{lemma}\label{lem: conti Qtnu}
$\mathrm{(i)}$ The set $\Qc$ is convex and the map $\Q \in \Qc \longmapsto J(t, \nu, \Q)$ is strictly concave {for each $(t,\nu)\in [0,T]\x \Pc_2(\R^d)$}.

\vspace{0.5em}

\noindent $\mathrm{(ii)}$ For each $(t, \nu)\in [0,T]\x \Pc_2(\R^d)$, there exists a unique probability $\Q^*_{t, \nu} \in \Qc$ such that
$$
U(t, \nu) = J(t, \nu, \Q^*_{t,\nu}).
$$
Moreover, the map $(t, \nu)\in [0,T]\x \Pc_2(\R^d) \longmapsto \Q^*_{t,\nu}$ is continuous.
\end{lemma}
\begin{proof}
$\mathrm{(i)}$ Fix $\Q_1, \Q_2 \in \Qc$ and define $\Q := (\Q_1 + \Q_2)/2$.
Then, there exist some processes $\alpha^1, \alpha^2 \in \Ac$ such that,
for any $\varphi \in C^2_b(\R^d)$, the process
$$
\varphi(B_{{\cdot}}) - \int_0^{{\cdot}} \Big( \alpha^i_s \cdot D \varphi(B_s) + \frac12 \Delta \varphi(B_s) \Big) \d s
~\mbox{is a}~\Q_i \mbox{-martingale},
~i=1, 2.
$$
 
Next, let $\M$ denote the space of all positive Borel measures $q$ on $[0,T] \x A$ such that the marginal distribution on $[0,T]$ is the Lebesgue measure, so that one can write $q(\d t, \d a) = q_t(\d a) \d t$, where $(q_t(da))_{t\le T}$ is a Borel measurable kernel from $[0,T]$ to $\Pc(A)$.
We introduce the enlarged canonical space $\Om \x \M$ with canonical element   denoted by $(\xi, W, B, \Lambda)$, {and simply} write $\Lambda (\d a, \d s) = \Lambda_s(\d a)\d s$ {whenever it is convenient}.
For  each  $i=1, 2$, we define the $\M$-valued random variable   $\Lambda^{\alpha^i}(\d a, \d s) := \delta_{\alpha^i_s } (\d a) \d s$,
and then let
$$
\Qb_i := \Q_i \circ \big(\xi, W, B,  \Lambda^{\alpha^i} \big)^{-1},
~i=1,2.
$$
It is then clear that,
for any $\varphi \in C^2_b(\R^d)$, the process
$$
\Sb^{\varphi}_{{\cdot}} := \varphi(B_{{\cdot}}) - \int_0^{{\cdot}} \!\! \int_A a \cdot D \varphi(B_s) \Lambda(\d a, \d s) - \int_0^{{\cdot}} \frac12 \Delta \varphi(B_s) \d s
~\mbox{is a}~\Qb_i \mbox{-martingale},
~i=1, 2.
$$
Moreover, since $A$ is convex, the process $\alphab_s := \int_A a \Lambda_s (\d a)$ is still $A$-valued, and, under each probability $\Qb_1$, $\Qb_2$ and $\Qb := (\Qb_1 + \Qb_2)/2$,
$$
B_{{\cdot}} - \int_0^{{\cdot}} \alphab_s \d s ~\mbox{is a standard Brownian motion.}
$$
It follows that {the process} $\alpha$ {defined by} $\alpha_s := \E^{\Qb}[ \alphab_s | B]${, $s\ge 0$,}
{is such that} $(B_t - \int_0^t \alpha_s \d s{)_{t\ge 0}}$ is a standard Brownian motion under $\Q = \Qb|_{\Om} = (\Q_1 + \Q_2)/2$.
As $A$ is convex, the process $\alpha$ is also $A$-valued, and hence $\Q := (\Q_1 + \Q_2)/2 \in \Qc$, so that $\Qc$ is convex.
\vspace{0.5em}
When $\Q_1 \neq \Q_2$,
one has $(\Q \otimes \lambda ) \{ (\om, t) {\in \Omega\x [0,T]}: \alpha^1_t(\om) \neq \alpha^2_t(\om) \} > 0${, in which $\lambda$ denotes the Lebesgue's measure,}
and hence $(\Qb \otimes {\lambda}) \big\{ (\omb,t) = (\om, q, t) {\in \Omega\x \M \x [0,T]} ~: \alphab_t(\omb) \neq \alpha_t(\om) \big\} > 0$.
By the strict concavity of $a \longmapsto L(t, \nu, a)$ and Jensen's inequality, it follows that
\begin{align*}
J (t, \nu, \Q)
>~&
\E^{\Qb} \Big[ \int_t^T L(s, \mu^{t,\nu}_s, \alphab_s) \d s + g( \mu^{t, \nu}_T) \Big] \\
=~&
\frac12 \sum_{i=1}^2 \E^{\Qb_i} \Big[ \int_t^T L(s, \mu^{t,\nu}_s, \alphab_s) \d s + g( \mu^{t, \nu}_T) \Big]
=
\frac{J(t, \nu, \Q_1) + J(t, \nu, \Q_2)}{2}.
\end{align*}
Therefore, the map $\Q \longmapsto J(t, \nu, \Q)$ is strictly concave.

\vspace{0.5em}

\noindent $\mathrm{(ii)}$
Let us now fix $(t, \nu)$, and let $(\Q_n)_{n \ge }$ be a maximizing sequence of the control problem, i.e.
$$
\lim_{n \to \infty} J(t, \nu, \Q_n) = U(t, \nu).
$$
Let $\alpha^n \in \Ac$ be the control process corresponding to $\Q_n$, and $\Lambda^{\alpha^n} (\d a, \d s) := \delta_{\alpha^n_s}(\d a) \d s$. Set 
  $\Qb_n := \Q_n \circ (B, \Lambda^{\alpha^n})^{-1}$.
Since $A$ is compact,  $(\Qb_n)_{n \ge 1}$ is relatively compact and hence, along some subsequence,
$$
\Qb_{n} \longrightarrow \Qb
~\mbox{weakly, for some}~
\Qb.
$$
One can further deduce that $\Sb^{\varphi}$ is a $\Qb$-martingale for any $\varphi \in C^2_b(\R^d)$,
so that $\Q := \Qb|_{\Om} \in \Qc$.
Therefore,
\begin{align*}
J(t, \nu, \Q)
\ge&~
\E^{\Qb} \Big[ \int_t^T L(s, \mu^{t,\nu}_s, \alphab_s) \d s + g( \mu^{t, \nu}_T) \Big] \\
=&~ \lim_{n \to \infty}
\E^{\Qb_n} \Big[ \int_t^T L(s, \mu^{t,\nu}_s, \alphab_s) \d s + g( \mu^{t, \nu}_T) \Big]
=
\lim_{n \to \infty} J(t, \nu, \Q_n).
\end{align*}
It follows that $\Q$ is the optimal control measure, i.e.~$J(t, \nu, \Q) = U(t, \nu)$.
Moreover, by the strict concavity of $\Q \longmapsto J(t, \nu, \Q)$, the optimal control measure $\Q \in \Qc$ is unique, which we denote by $\Q_{t, \nu}$.
\vspace{0.5em}
Let us now consider a sequence $(t_n, \nu_n)_{n \ge 1}$ such that $(t_n, \nu_n) \longrightarrow (t, \nu)$.
By the compactness of $A$ again, the sequence of optimal control measure $(\Q_{t_n, \nu_n})_{n \ge 1}$ is relatively compact.
Taking an arbitrary convergent subsequence such that $\lim_{k \to \infty} \Q_{t_{n_k}, \nu_{n_k}} = \Q$, as above, one can deduce that $\Q \in \Qc$.
Let $\Qb_{t_n, \nu_n}$ and $\Qb$ be the corresponding measures on the enlarged space. {Recall from Lemma \ref{lemm:mu_Lip} that}
${(t,\nu)\mapsto}\mu^{t,\nu}_s$ is Lipschitz (and hence uniformly continuous) in $(t, \nu)${, for each $s\le T$}. It follows that
\begin{align*}
U(t, \nu)
\ge
J(t, \nu, \Q)
&= \E^{\Qb} \Big[ \int_t^T L(s, \mu^{t,\nu}_s, \alphab_s) \d s + g( \mu^{t, \nu}_T) \Big] \\
&=
\lim_{k \to \infty} \E^{\Qb_{t_{n_k}, \nu_{n_k}}} \Big[ \int_{{t_{n_k}}}^T L(s, \mu^{t,\nu}_s, \alphab_s) \d s + g( \mu^{t, \nu}_T) \Big] \\
&=
\lim_{k \to \infty} \E^{\Qb_{t_{n_k}, \nu_{n_k}}} \Big[ \int_{{t_{n_k}}}^T L(s, \mu^{t_{n_k},\nu_{n_k}}_s, \alphab_s) \d s + g( \mu^{t_{n_k},\nu_{n_k}}_T) \Big] \\
&=
\lim_{k \to \infty} J \big(t_{n_k}, \nu_{n_k}, \Q_{t_{n_k}, \nu_{n_k}} \big)
=
\lim_{k \to \infty} U(t_{n_k}, \nu_{n_k}).
\end{align*}
At the same time, one has
$$
U(t, \nu)
= J(t, \nu, \Q_{t,\nu})
= \lim_{k \to \infty} J \big( t_{n_k}, \nu_{n_k}, \Q_{t, \nu} \big)
\le
\lim_{k \to \infty} U(t_{n_k}, \nu_{n_k}).
$$
Therefore, one has the equality $U(t, \nu) = J(t, \nu, \Q)$, and hence $\Q$ coincides with the unique optimal control measure with initial condition $(t, \nu)$, i.e. $ \Q = \Q_{t, \nu}$.
This is enough to show that $\lim_{n \to \infty} \Q_{t_n, \nu_n} = \Q_{t, \nu}$.
\end{proof}
 
\vspace{0.5em}
 
We next study the differentiability of $\nu \longmapsto J(t, \nu, \Q)$,
for simplicity, we will present the result in the $d=1$ dimensional setting.
Recall that, given $t\le T${, $x \in \R$} and a random variable $\xi \sim \nu$, the process $X^{t, \xi}$ and $X^{t, x, \xi}$ are defined by
$$
X^{t,\xi}_s = \xi + \int_t^s b(r, X^{t,\xi}_r, \mu^{t,\xi}_r) \d r+ \int_t^s \sigma(r, X^{t,\xi}_r, \mu^{t,\xi}_r) \d W_r + \int_t^s \sigma_0 \d B_r,
~s \in [t,T],
$$
with $\mu^{t,\xi}_s := \Lc(X^{t,\xi}_s | B)$ and
$$
X^{t,x,\xi}_s = x + \int_t^s \!\! b(r, X^{t,x,\xi}_r, \mu^{t,\xi}_r) \d r+ \int_t^s \!\! \sigma(r, X^{t,x, \xi}_r, \mu^{t,\xi}_r) \d W_r + \int_t^s \sigma_0 \d B_r,
~s \in [t,T].
$$
Notice also that $X^{t,\xi} = X^{t,x,\xi}|_{x = \xi}$.
\vspace{0.5em}
Let us first introduce the tangent process $\partial_x X^{t,x,\xi}$ of the process $X^{t,x,\xi}$ defined by
\begin{align*}
\partial_x X^{t,x,\xi}_s
=
{1}
+ \int_t^s \!\! \partial_x b(r, X^{t,x,\xi}_r, \mu^{t,\xi}_r) \partial_x X^{t,x,\xi}_r \d r
& + \int_t^s \!\! \partial_x \sigma(r, X^{t,x,\xi}_r, \mu^{t,\xi}_r) \partial_x X^{t,x,\xi}_r \d W_r,
~s \in [t,T].
\end{align*}
Next, we introduce the tangent process of $X^{t,\xi}$ w.r.t. $\nu := \Lc(\xi)$ (see \cite[Proposition 4.1]{BuckdahnLiPeng}):
\begin{align*}
U^{t,\xi}_s(y)
=&
\int_t^s \partial_x b(r, X^{t,\xi}_r, \mu^{t,\xi}_r) U^{t,\xi}_r (y) \d r
+ \int_t^s \partial_x \sigma(r, X^{t,\xi}_r, \mu^{t,\xi}_r) U^{t,\xi}_r(y) \d W_r \\
&+
\int_t^s \Et
\Big[
D_m b (r, X^{t,\xi}_r, \mu^{t,\xi}_r, \Xt^{t,y,\xit}_r) \partial_x \Xt^{t, y, \xit}_r
+
D_m b(r, X^{t,\xi}_r, \mu^{t,\xi}_r, \Xt^{t, \xit}_r) \Ut^{t, \xit}_r(y)
\Big] \d r \\
&+
\int_t^s \Et
\Big[
D_m \sigma (r, X^{t,\xi}_r, \mu^{t,\xi}_r, \Xt^{t,y,\xit}_r) \partial_x \Xt^{t, y, \xit}_r
+
D_m \sigma (r, X^{t,\xi}_r, \mu^{t,\xi}_r, \Xt^{t, \xit}_r) \Ut^{t, \xit}_r(y)
\Big] \d W_r,
\end{align*}
where $(\xit, \Wt)$ is an independent copy (with the same distribution) of $(\xi, W)$, such that $(\xit, \Wt)$ is independent of $(\xi, W, B)$,
the process $(\Xt^{t,y,\xit}, \Xt^{t,\xit})$ is defined as $(X^{t,y,\xi}, X^{t,\xi})$ but with $(\xit, \Wt, B)$ in place of $(\xi, W, B)$,
and $\Et [\cdot] := \E[ \cdot | \xi, B, W]$.
 
\begin{lemma} \label{lemm:DmJ}
For each fixed $\Q \in \Qc$ {and $t\le T$}, the map $\nu {\in \Pc_2(\R^d)} \mapsto J(t, \nu, \Q)$ is differentiable and
\begin{align*}
D_m J(t, \nu, \Q, {x})
=
\E^{\Q} \Big[
&\int_t^T \Big( D_m L(s, \mu^{t, \nu}_s, X^{t, x, \xi}_s) \partial_x X^{t, x, \xi}_s + D_m L(s, \mu^{t, \nu}_s, X^{t, \xi}_s) U^{t, \xi}_s(x) \Big) \d s \\
&~~~{+} D_m g(\mu^{t, \nu}_T, X^{t, x, \xi}_T) \partial_x X^{t, x, \xi}_T + D_m g(\mu^{t, \nu}_T, X^{t, \xi}_T) U^{t, \xi}_T (x)
\Big], \; {x\in \R^d}.
\end{align*}
Moreover, the map
$$
(t, \nu, \Q, x){\in [0,T]\x \Pc_2(\R^d)\x \Qc \x \R^d} \longmapsto D_m J(t, \nu, \Q, x)
~\mbox{ is uniformly continuous}.
$$
\end{lemma}
The proof is almost the same that in \cite[Theorem 4.2 and Lemma 6.1]{BuckdahnLiPeng}, and is omitted here.
{We are now in a position to complete the proof of Theorem \ref{thm:regularity}.}

\vspace{0.5em}
 
\begin{proof}[Proof of Theorem \ref{thm:regularity}]
Let $\nu$, $\nu' \in \Pc_2(\R^d)$ be such that $\nu \neq \nu'$. We set $\nu_{\eps} := \nu + \eps (\nu' - \nu)$.
Then, {with the notations of Lemma \ref{lem: conti Qtnu},}
\begin{align*}
U (t, \nu) - U(t, \nu_{\eps})
~\le~ &
J(t, \nu, {\Q^*_{t, \nu}}) - J(t, \nu_{\eps},{\Q^*_{t, \nu}}) \\
~=~ &
\int_0^1 \int_{\R^d} { \delta_m J \big(t, \lambda \nu + (1 - \lambda) \nu_{\eps}, {\Q^*_{t, \nu}}, x \big)} \big[ \nu - \nu_{\eps} \big] \d x \d \lambda.
\end{align*}
Similarly,
$$
U(t, \nu) - U(t, \nu_{\eps})
~\ge~
\int_0^1 \int_{\R^d} { \delta_m J \big(t, \lambda \nu + (1 - \lambda) \nu_{\eps}, {\Q^*_{t, \nu_\eps}}, x \big) } \big[ \nu - \nu_{\eps} \big] \d x \d \lambda.
$$
Using the continuity of $(t, \nu) \longmapsto {\Q^*_{t, \nu}}$, {see Lemma \ref{lem: conti Qtnu}}, it follows that
$$
\lim_{\eps \searrow 0} \frac{U (t, \nu_{\eps}) - U(t, \nu) }{\eps}
=
\int_{\R^d} {  \delta_m J \big(t, \nu, {\Q^*_{t, \nu}}, x \big)} \big[ \nu' - \nu \big]\d x,
$$
which implies that
$$
\delta_m U(t, \nu, x) = {  \delta_m J \big(t, \nu, {\Q^*_{t, \nu}}, x\big) }, ~\mbox{for}~\nu \mbox{-a.e.}~ x \in \R^d.
$$
It follows that
$$
D_m U(t, \nu, x) = { D_m J \big(t, \nu, {\Q^*_{t, \nu}}, x \big)}, ~\mbox{for}~\nu \mbox{-a.e.}~ x \in \R^d.
$$
One can then conclude the proof with {Lemmas \ref{lemm:DmJ} and \ref{lem: conti Qtnu}.}
\end{proof}

\subsection{$C^2$-regularity and weak convergence rate}
 
When both volatility coefficients $\sigma$ and $\sigma_0$ are constant,
one has a better structure  which enables us to prove that the solution to the master equation of the mean-field control problem is indeed $\Cc^2$.
This allows us to obtain a weak convergence rate by standard techniques.

\begin{assumption} \label{assum:C2}
$\mathrm{(i)}$ Both coefficient functions $\sigma$ and $\sigma_0$ are constant,
and $b(\cdot, a) \equiv a$ for all $a\in A$. The constant $\sigma_0$ is non-degenerate.

\vspace{0.5em}
\noindent $\mathrm{(ii)}$ The reward function is given by $L(t, \mu, a) = L_0 (t, a) + F(\mu)$, $(t,\mu,a)\in [0,T]\x \Pc^2(\R^d)\x A$,
for some $L_0: [0,T]\x A \longrightarrow \R$ and $F: \Pc_2(\R^d) \longrightarrow \R$, so that
$$
H(t, \mu, z) = H_0(t, z) + F(\mu),
~~\mbox{with}~~
H_0(t, z) := \sup_{a \in A} \big( L_0(t, a) + a \cdot z \big),
$$
for all $(t,\mu,z)\in [0,T]\x \Pc^2(\R^d)\x \R^d$.

\vspace{0.5em}
\noindent $\mathrm{(iii)}$
The functions $F$, $D_m F$, $D^2_{mm} F$, $g$, $D_m g$, and $D^2_{mm} g$, and $H_0$, $D_z H_0$ and $D^2_{zz} H_0$ are all bounded
and Lipschitz continuous.
\end{assumption}

In this setting, the controlled process is defined by
$$
X^{\alpha}_t = X_0 + \sigma W_t + X^{0, \alpha}_t,
~~\mbox{with}~~
X^{0, \alpha}_t := \int_0^t \alpha_s \d s +\sigma_0 B_t,
$$
so that, given $\alpha \in \Ac_0$,
\begin{equation} \label{eq:mu_a_struc}
\mu^{\alpha}_t
~:=~
\Lc \big(X^{\alpha}_t \big| B \big)
~=~
\nu_0 \star N \big( x, \sigma \sigma^{\top} t \big) \big|_{x = X^{0, \alpha}_t},
\end{equation}
where the last term $\nu_0 \star N \big( x, \sigma \sigma^{\top} t \big)$ means the probability measure obtained by the convolution of the distribution $\nu_0 \in \Pc_2(\R^d)$ and the Gaussian distribution $N(x, \sigma \sigma^{\top} t)$ on $\R^d$.
In other words, given the initial distribution $X_0 \sim \nu_0$ and $\sigma$,
the conditional distribution $\mu^{\alpha}_t$ is completely characterized by $t$ and $X^{0, \alpha}_t$,
so that the mean-field problem \eqref{eq:VJQ} reduces to a classical optimal control problem,
and the corresponding value function $U(0, \nu_0)$ can be characterized by a classical HJB equation.
Let us define
$$
\widetilde L(t, x, a; \nu_0) := L_0(t, a) + \widetilde{F} (t, x, \nu_0),
$$
with
$$
\widetilde{F} (t, x, \nu_0) := F \big(\nu_0 \star N(x, \sigma \sigma^{\top} t) \big)
~\mbox{and}~
\widetilde{G} (x, \nu_0) := g \big(\nu_0 \star N(x, \sigma \sigma^{\top} T) \big),
$$
for all $(t,x,a,\nu_0)\in [0,T]\x \R^d\x A\x \Pc_2(\R^d)$.

\begin{remark}
Under the regularity conditions on $F$ and $g$ in Assumption \ref{assum:C2}, it is easy to deduce the regularity of $\widetilde{F}$ and $\widetilde{G}$.
In particular, given two independent Brownian motions $W$ and $\widetilde W$ and $\xi \in \Fc_0$ such that $\xi \sim \nu_0$, one deduces from the definition of $D_m$ and $D^2_{mm}$ and direct computations that   
\begin{align*}
D_m \widetilde F(t, x, \nu_0{,}y)
&=
\E \Big[ D_m F \big(\nu_0 \star N(x, \sigma \sigma^{\top} t), ~x+ \sigma W_t + y \big) \Big], \\
D_m \widetilde{G} \big(x, \nu_0{,} y \big)
&=
\E \Big[ D_m g \big(\nu_0 \star N(x, \sigma \sigma^{\top} T), ~x+ \sigma W_T + y \big) \Big], 
\end{align*}
and
\begin{align*}
D^2_{mm} \widetilde F(t, x, \nu_0{,} y, z)
&=
\E \Big[ D^2_{mm} F \big(\nu_0 \star N(x, \sigma \sigma^{\top} t), ~x+ \sigma W_t + y, x+ \sigma \widetilde W_t + z \big) \Big], \\
D^2_{mm} \widetilde{G} \big(x, \nu_0{,} y, z \big)
&=
\E \Big[ D^2_{mm} g \big(\nu_0 \star N(x, \sigma \sigma^{\top} T), ~x+ \sigma W_T + y, ~x+ \sigma \widetilde W_T + z \big) \Big],
\end{align*}
which are all bounded and continuous under Assumption \ref{assum:C2}.
Moreover,
\begin{equation} \label{eq:DFG_Lip}
x \longmapsto \Big(\widetilde F(t,x,\nu_0), \widetilde G(x, \nu_0), D_m \widetilde F(t,x,\nu_0{,}y), D_m \widetilde G(x, \nu_0{,} y) \Big)
\end{equation}
{is Lipschitz uniformly in all variables.
We also have the expressions
\begin{equation} \label{eq:DxFG}
 \begin{cases}
	D_x \tilde F(t, x, \nu_0) = \E \Big[ D_m F \big(\nu_0 \star N(x, \sigma \sigma^{\top} t), ~x+ \sigma W_t + \xi \big) \Big],\\
	D_x \tilde G(x, \nu_0) = \E \Big[ D_m g \big(\nu_0 \star N(x, \sigma \sigma^{\top} T), ~x+ \sigma W_T + \xi \big) \Big].
\end{cases}
\end{equation}
}

\end{remark}
\begin{lemma} \label{lemm:HJB_x}
Let Assumptions \ref{assum:convexity} and \ref{assum:C2} hold true, and $\nu_0 \in \Pc_2(\R^d)$ be fixed.
Then, the HJB equation
\begin{equation} \label{eq:HJB_tx}
\partial_t \Ut
+
H_0 \big(\cdot, D_x \Ut \big)
+
\frac12 \mathrm{Tr} \big[ \sigma_0 \sigma_0^{\top} D^2_{xx} \Ut \big]
+
\widetilde{F} \big(\cdot, \nu_0 \big) = 0 \mbox{ on }  [0,T)\x \R^d,
\end{equation}
with terminal condition
\begin{equation} \label{eq:HJB_tx_termc}
\Ut(T, \cdot) = \widetilde{G} \big(\cdot, \nu_0 \big) \mbox{ on }  \R^d
\end{equation}
has a unique solution  $\Ut (\cdot; \nu_0)$ {with polynomial growth}. 
Moreover, $D_x \Ut(\cdot; \nu_0)$ is bounded continuous and $x\in \R^d \mapsto D_x \Ut(t,x; \nu_0)$ is   Lipschitz, uniformly in $(t,x) \in [0,T]\x \R^d$.
Further, the value function $U$ in \eqref{eq:VJQ} is given by
\begin{equation} \label{eq:reform_U2Ut}
U({t}, \nu_0) = \Ut({t}, 0; \nu_0),
~~\mbox{for all}~~
\nu_0 \in \Pc_2(\R^d)\mbox{ and } t\le T.
\end{equation}
\end{lemma}
 
\begin{proof}
Using the equality \eqref{eq:mu_a_struc}, it follows immediately that the mean-field control problem \eqref{eq:VJQ}
can be reformulated as a classical optimal control problem: 
\begin{align*}
U(0, \nu_0)
=
\Ut (0, x;\nu_0) |_{x = 0}
\end{align*}
{where}
\begin{align*}
\Ut (t, x; \nu_0) 
&:=
\sup_{\alpha \in \Ac_0}
\E \Big[{
\int_{{t}}^T \Lt \big(s, x+ X^{ \alpha}_s-X^{ \alpha}_t , \alpha_s; \nu_0 \big) \d s + \Gt \big( x+X^{ \alpha}_T-X^{  \alpha}_t; \nu_0 \big)}
\Big].
\end{align*}
Then, for fixed $\nu_0 \in \Pc_2(\R^d)$, it is classical (see e.g.~\cite{BouchardTanWang2}) to prove that $(t,x) \longmapsto \Ut(t,x, \nu_0)$  is the unique classical solution with polynomial growth  to the HJB equation \eqref{eq:HJB_tx} with terminal condition $\widetilde{G}$, {so that \eqref{eq:reform_U2Ut} holds.}
Further, appealing to the finite difference technique as in the proof of Theorem \ref{thm:UC2} below, it is easy to obtain that 
$$
	{D_x}\Ut(t,x, \nu_0) = \E \Big[ \int_t^T {D_x}\widetilde{F} \big(s, X^{t,x}_s, \nu_0 \big) \d s + {D_x} \widetilde{G} \big(X^{t,x}_T, \nu_0) \Big],
$$
where $X^{t,x}$ is the unique solution to the SDE
$$
	X^{t,x}_s = x + \int_t^s D_z H_0(r, D_x \Ut(r, X^{t,x}_r)) \d r + \sigma_0 (W_s - W_t), ~s \ge t,
$$
for some Brownian motion $W$. 
In view of \eqref{eq:DxFG} and Assumption \ref{assum:C2},  one deduces that $D_x \Ut$ is bounded continuous, uniformly in all its variables.
Similarly, one can deduce that $D^2_{xx} \Ut$ is   bounded continuous, uniformly in all its variables, so that $x \mapsto D_x \Ut(t,x; \nu_0)$ is Lipschitz, uniformly in $(t,x) \in [0,T]\x \R^d$.
\end{proof}

\paragraph{$C^2$-regularity.}
 
In view of Lemma \ref{lemm:HJB_x}, studying the regularity of $\nu_0 \longmapsto U(0, \nu_0)$ reduces to studying the regularity of $\Ut(0,0; \nu_0)$ w.r.t. the parameter $\nu_0$.

\begin{theorem} \label{thm:UC2}
Let Assumptions \ref{assum:convexity} and \ref{assum:C2} hold true.
Then,   $D_m U(\cdot, \nu_0)$ and $D^2_{mm} U(\cdot, \nu_0)$ exist, are bounded and continuous. 
\end{theorem}
\begin{proof}
$\mathrm{(i)}$
For fixed $(\nu_0, y) \in \Pc_2(\R^d) \x \R^d$ { and $i \in \{1, \ldots, d\}$,}
let us define the map $(t,x) \longmapsto \Ut^i_1(t,x; \nu_0)$ as the unique solution with polynomial growth to the (linear parabolic) PDE
\begin{align} \label{eq:PDE_Ut1}
\partial_t \Ut^i_1(\cdot;\nu_0)
+
D_z H_0 \big(\cdot, D_x \Ut(\cdot, \nu_0) \big) \cdot D_x \Ut^i_1 (\cdot;\nu_0)
&+
\frac12 \mathrm{Tr} \big[ \sigma_0 \sigma_0^{\top} D^2_{xx} \Ut^i_1 (\cdot;\nu_0) \big] \nonumber \\
&+ D_{m,i} \widetilde F(\cdot, \nu_0; y)
= 0 \mbox{ on $[0,T)\x \R^d$},
\end{align}
with terminal condition $\Ut^i_1(T,\cdot;\nu_0) = D_{m,i} \widetilde{G} \big(\cdot, \nu_0; y \big)$ on $\R^d$, and where $D_{m,i} \widetilde{F}$ (resp.~$D_{m,i} \widetilde{G}$) denotes the $i$-th component of the vector $D_m \widetilde{F}$ (resp.~$D_{m} \widetilde{G}$). By the uniform Lipschitz continuity of $D_x\Ut(\cdot;\nu_0)$ in Lemma \ref{lemm:HJB_x}, standard existence and comparison results imply that  \eqref{eq:PDE_Ut1} has a unique classical solution $\Ut^i_1$ which is bounded (see e.g.~\cite{lieberman1996second}).  We also know that    $(D_x \Ut^i_1,D^2_{xx}\Ut^{i}_1)(\cdot;\nu_0)$ is H\"older continuous on each compact set of $[0,T)\x \R^d$.
Moreover,   the Lipschitz continuity of $x \mapsto \big( D_x H(t, D_x \Ut(t,x, \nu_0)), ~D_{m,i} \widetilde{F}(t,x, \nu_0) \big)$ imply that $x \mapsto \Ut^i_1 (t,x, \nu_0)$ is Lipschitz uniformly in all variables, so that $D_x \Ut^i_1 $ is uniformly bounded.

\vspace{0.5em}

Next, for fixed $(\nu_0, y, z) \in \Pc_2(\R^d) \x \R^d \x \R^d$ { and $i, j \in \{1, \ldots, d\}$,} let us define the map $(t,x) \longmapsto \Ut^{i,j}_2(t,x; \nu_0)$ as the unique solution with polynomial growth to the (linear parabolic) PDE
\begin{align}  \label{eq:PDE_U2}
\partial_t &\Ut^{i,j}_2 (\cdot;\nu_0)
+
D_z H_0 \big( \cdot, D_x \Ut(\cdot; \nu_0) \big) D_x \Ut^{i,j}_2 (\cdot;\nu_0)
+
\frac12 \mathrm{Tr} \big[ \sigma_0 \sigma_0^{\top} D^2_{xx} \Ut^{i,j}_2(\cdot;\nu_0)  \big] \nonumber \\
&+
\big \langle D^2_{zz} H_0 \big( \cdot, D_x \Ut(\cdot; \nu_0) \big) D_x \Ut^i_1(\cdot;\nu_0),  D_x \Ut^j_1 (\cdot;\nu_0) \big \rangle
+
D^2_{mm, i, j} \widetilde{F} \big(t,x, \nu_0; y, z \big)
=
0, \mbox{ on $[0,T)\x \R^d$},
\end{align}
with terminal condition $\Ut_2(T, \cdot;\nu_0) = D^2_{mm, i, j} \widetilde{G} \big(\cdot, \nu_0; y, z \big)$ on $\R^d$,
where $D^2_{mm, i, j} \widetilde{F}$ (resp.~$D^2_{mm, i, j} \widetilde{G})$ denotes the $(i,j)$-element of the matrix $D^2_{mm} \widetilde{F}$ (resp.~$D^2_{mm} \widetilde{G}$).
 As above,   the H\"older continuity of $ (D_x \Ut(, \nu_0),D_x \Ut^i_1(\cdot; \nu_0),D_x \Ut^j_1(\cdot; \nu_0))$ implies that the PDE \eqref{eq:PDE_U2} has a unique classical solution $\Ut^{i,j}_2$, which is in addition uniformly bounded as $D^2_{zz}H_0$, $D_x \Ut^i_1$, $D_x \Ut^j_1$ and $D^2_{mm} \widetilde{F}$ are all uniformly bounded. 

\vspace{0.5em}

It remains to prove that 
$$
	D_m \Ut = \big( \Ut^1_1, \ldots, \Ut^d_1)
	~~\mbox{and}~~
	D^2_{mm} \Ut = \big( \Ut^{i,j}_2 \big)_{1 \le i, j \le d}.
$$
To this purpose, let us first lift $\nu \mapsto (\widetilde U, \widetilde F, \widetilde G)(\cdot, \nu)$ as a functional defined on the space $\L^2(\Fc_0)$ of square integrable random variables: 
$$
\widehat U(\cdot; \xi) := \widetilde U \big(\cdot; \Lc(\xi) \big),
~~
\widehat F(\cdot; \xi) := \widetilde F\big(\cdot; \Lc(\xi)\big)
~~\mbox{and}~~
\widehat G(\cdot; \xi) := \widetilde G\big(\cdot; \Lc(\xi)\big),\mbox{ for  $\xi \in \L^2(\Fc_0)$,}
$$
For $\xi, ~\eta \in \L^2(\Fc_0)$ s.t. $\xi \sim \nu_0$ and $h > 0$, set
$$
\Delta_h \widehat U (t,x) := \frac1h \Big( \widehat U \big(t, x; \xi + h \eta \big) - \widehat U \big(t, x; \xi \big) \Big),\;(t,x)\in [0,T]\x \R^d,
$$
and then define $\Delta_h \widehat F $ and $\Delta_h \widehat G $ similarly.
Then $(t,x) \longmapsto \Delta_h \widehat U(t,x)$ solves the linear PDE
\begin{align*}
\partial_t \Delta_h \widehat U  
+
{\Xi_h} D_x \Delta_h \widehat U
+
\frac12 \mathrm{Tr} \big[ \sigma_0 \sigma_0^{\top} D^2_{xx} \Delta_h \widehat U \big]
+
\Delta_h \widehat F=0
 \mbox{ on $[0,T)\x \R^d$}
\end{align*}
where 
$$
{\Xi_h := \int_0^1 D_z H_0 \big(\cdot, \lambda D_x\widehat{U}(\cdot;\xi+h\eta)+(1-\lambda) D_x \widehat U(\cdot; \xi) \big) d\lambda}
$$
with terminal condition
$$
\Delta_h \widehat U (T, \cdot) = \Delta_h \widehat G \mbox{ on $\R^d$.}
$$
By stability of   \eqref{eq:HJB_tx}-\eqref{eq:HJB_tx_termc}, it follows that
$$
\Xi_h  \longrightarrow  D_z H_0 \big(\cdot,   D_x\widehat{U}(\cdot;\xi) \big)  = D_z H_0 \big(\cdot,  D_x \Ut(\cdot; \nu_0) \big),
~~\mbox{as}~
h \searrow 0.
$$
Moreover, since for $(t,x)\in [0,T]\x \R^d$ 
$$
\Big( \Delta_h \widehat F (t,x),~ \Delta_h \widehat G(x) \Big)
~\longrightarrow~
\Big( \E \big[ D_m \widetilde F(t,x, {\nu_0}; \xi) \cdot \eta \big], ~\E \big[ D_m \widetilde G(x, {\nu_0}; \xi) \cdot \eta \big] \Big),
~~\mbox{as}~
h \searrow 0,
$$
it follows that $\Delta_h \widehat U$ converges pointwise to some  $\widehat U_1$ which   solves 
\begin{align*}
\partial_t \widehat U_1 
+
D_z H_0 \big(\cdot, D_x \Ut(\cdot, \nu_0) \big) D_x \widehat U_1
+
\frac12 \mathrm{Tr} \big[ \sigma_0 & \sigma_0^{\top} D^2_{xx} \widehat U_1  \big] \nonumber \\
&+ \E \big[ D_m \widetilde F(\cdot,{\nu_0;\xi}) \cdot \eta \big]
= 0, \mbox{ on $[0,T)\x \R^d$}
\end{align*}
with terminal condition $\widehat U_1(T,\cdot) = \E \big[ D_m \widetilde G(\cdot,{\nu_0}; \xi) \cdot \eta \big]$   on $\R^d$.
 Again, by the Lipschitz continuity of $x \mapsto D_x \Ut(t,x, \nu_0)$ in Lemma \ref{lemm:HJB_x} and the boundedness of $D_m \widetilde{F}$ and $\D_m \widetilde{G}$,
 $\widehat U_1$ is uniformly bounded.

By the above, \eqref{eq:PDE_Ut1} and the abritrariness of $y \in \R^d$, we have proved that, {for all fixed $(t,x)\in [0,T]\x \R^d$}, 
$$
D_m \Ut(t,x, \nu_0; y) = \Ut_1(t,x; \nu_0; y),
~\mbox{for}~\nu_0 \mbox{-a.e.}~y \in \R^d.
$$
Similarly, one can deduce that
$$
D^2_{mm} \Ut(t,x, \nu_0; y, z) = \Ut_2(t,x; \nu_0; y, z) {~\mbox{for}~\nu_0\times \nu_0 \mbox{-a.e.}~(y,z) \in \R^{2d}.}
$$
Together with \eqref{eq:reform_U2Ut} and the fact that {$\Ut_1$ and $\Ut_2$ are continuous and bounded}, this concludes the proof.
\end{proof}

\paragraph{The weak convergence rate.}
 
Using the regularity of the value function $U$, one can easily deduce a weak convergence rate,
which is the optimal rate as it is the same as in the linear setting.
  For  $x = (x_1, \ldots, x_N)$, let us define $m^x := \frac1N \sum_{k=1}^N \delta_{x_k}$.

\begin{theorem}
Let Assumptions \ref{assum:convexity} and \ref{assum:C2} hold true. 
Then there exist a constant $C>0$ such that
$$
\big| U(0, m^x) - V^N(0, x_1, \ldots, x_N) \big|
\le
C/N
$$
{for all $x=(x_1, \ldots, x_N)\in (\R^d)^N$, $N\ge 1$.}
\end{theorem}
\begin{proof}
Set
$$
U^N(t, x) := U(t, m^x),
$$
so that, for all $k \neq \ell$,
$$
	D_{x_k} U^N(t, x) = \frac1N D_m U(t, m^x, x_k),
	~~
	D^2_{x_k, x_{\ell}} U^N(t, x) = \frac{1}{N^2} D^2_{mm} U(t, m^x, x_k, x_{\ell}),
$$
and 
$$
	D^2_{x_k, x_k} U^N(t, x) = \frac1N \partial_x D_m U(t, m^x, x_k) + \frac{1}{N^2} D^2_{mm} U(t, m^x, x_k, x_k).
$$
Recalling the master equation \eqref{eq : Master eq U} satisfied by $U$,
it is easy to deduce that $U^N$ satisfies
\begin{align*}
\partial_t U^N(t, x) &+ \sup_{a \in A} \Big( L(t, m^x, a) + \sum_{k=1}^N b(t, x_k, m^x, a) \cdot D_{x_k} U^N(t, x) \Big) \\
&+ \frac12  \sum_{k=1}^N \mathrm{Tr} \big( \sigma\sigma^{\top} D^2_{x_k x_k} U^N(t, x) \big)
+ \frac12 \sum_{k, \ell =1}^N \mathrm{Tr} \big(  \sigma_0 \sigma_0^{\top} D^2_{x_k x_{\ell}} U^N(t,x) \big)
+ E_N(t,x)
= 0,
\end{align*}
where
$$
E_N(t,x) := - \frac{1}{2 N^2} \sum_{k=1}^N \mathrm{Tr} \big(  \sigma \sigma^{\top} D^2_{mm} U(t, m^x, x_k, x_k) \big).
$$
Recall that $V^N$ solves \eqref{eq:PDE_VN}-\eqref{eq:PDE_VN_term}. The result then follows by comparison of PDEs, using the bound on $D^2_{mm} U$ obtained in  Theorem \ref{thm:UC2}.
\end{proof}


\section{Application to  optimal control under partial observation}
\label{sec:partial_observ}

Optimal control problems under partial observation have been studied since the 1970s, using both the dynamic programming and the maximum principal approaches to characterize the value function or the optimal control.
However, its numerical approximation has been seldomly studied.
By applying the convergence results of the preceding sections, one can actually construct a controlled particle system on which a numerical approximation scheme can be based.
 
\paragraph{The optimal control problem under partial observation.}
Let us consider the following partial observation control problem
 formulated on a probability space $(\Om^*, \Fc^*, \P^*)$ equipped with two independent standard Brownian motion $(W, W^0)$ and a random variable $\xi \sim \nu_0$,
where the controlled state process $X$ follows the dynamic
$$
X^{\alpha}_t = \xi + \int_0^t b(s, X^{\alpha}_s, \alpha_s) \d s + \int_0^t \sigma(s, X^{\alpha}_s, \alpha_s) \d W_s + \int_0^t \sigma_0(s, X^{\alpha}_s) d W^0_s,
~t \le T,
~\P^*\mbox{-a.s.},
$$
in which the control process $\alpha$ is an adapted functional of the observable process $B^{\alpha}$ defined by 
$$
B^{\alpha}_t = \int_0^t h(s, X^{\alpha}_s) ds + W^0_t,\; t\le T.
$$
Let us denote by $\Ac_B$ the collection of all such control processes, then the value function is defined as
\begin{equation} \label{eq:V_partial}
V_P
~:=~
\sup_{\alpha \in \Ac_B} \E^{\P^*} \Big[
\int_0^T L \big(t, X^{\alpha}_t, \alpha_t \big) dt + g \big(X^{\alpha}_T \big)
\Big].
\end{equation}
Following the classical reference probability approach (see e.g. \cite{Bensoussan} or \cite[Section 2]{BuckdahnLiMa}), we consider a new probability $\Q^{\alpha}$ defined by
\begin{equation} \label{eq:SDE_Z_alpha}
\frac{\d \Q^{\alpha}}{\d \P^*} := (Z^{\alpha}_T)^{-1},
~~\mbox{with}~Z^{\alpha}~\mbox{defined by}~
\d Z^{\alpha}_t = h(t, X^{\alpha}_t) Z^{\alpha}_t \d B^{\alpha}_t,
~~
Z^{\alpha}_0 :=1.
\end{equation}
Then, $W$ and $B^{\alpha}$ are two independent standard Brownian motions under $\Q^{\alpha}$, so that the $Z^{\alpha}$ follows the dynamic in \eqref{eq:SDE_Z_alpha},  
$X^{\alpha}$  satisfies
\begin{align*}
\d X^{\alpha}_t
&=
\big( b(t, X^{\alpha}_t, \alpha_t) - \sigma_0 h(t, X^{\alpha}_t) \big) \d t
+
\sigma(t, X^{\alpha}_t, \alpha_t) \d W_t
+
\sigma_0(t, X^{\alpha}_t) \d B^{\alpha}_t,
\end{align*}
and  $V_P$ admits the equivalent formulation
$$
V_P
~=
\sup_{\alpha \in \Ac_B} \E^{\Q^{\alpha}} \Big[
\int_0^T Z^{\alpha}_t L \big(t, X^{\alpha}_t, \alpha_t \big) dt + Z^{\alpha}_T g \big(X^{\alpha}_T \big)
\Big].
$$

Since $W$ and $B^{\alpha}$ are two independent standard Brownian motions under $\Q^{\alpha}$, one can actually reformulate equivalently this control problem on the canonical space $(\Om, \Fc, \P)$ with two standard Brownian motions $(W, B)$, and random variable $\xi \sim \nu_0$, so that
\begin{equation} \label{eq:V_partial_reform}
V_P
~=
\sup_{\alpha \in \Ac_0} \E \Big[
\int_0^T Z^{\alpha}_t L \big(t, X^{\alpha}_t, \alpha_t \big) dt + Z^{\alpha}_T g \big(X^{\alpha}_T \big)
\Big],
\end{equation}
where the controlled process $(X^{\alpha}, Z^{\alpha})$ is defined by
\begin{align*}
\d X^{\alpha}_t
&=
\big( b(t, X^{\alpha}_t, \alpha_t) - \sigma_0 h(t, X^{\alpha}_t) \big) \d t
+
\sigma(t, X^{\alpha}_t, \alpha_t) \d W_t
+
\sigma_0(t, X^{\alpha}_t) \d B_t, \\
\d Z^{\alpha}_t
&=
h(t, X^{\alpha}_t) Z^{\alpha}_t \d B_t,
\end{align*}
with initial condition $X^{\alpha}_0 = \xi$ and $Z^{\alpha}_0 = 1$.

\paragraph{The approximative particle system.} 
{For the numerical approximation of the partial observation problem, it suffices to apply the results of the above section. 
Indeed, the equivalent reformulation \eqref{eq:V_partial_reform} falls into the class of   mean-field control problems of the form \eqref{eq:def_V}.} Hence, we first introduce the following particle system:
let $(B, (W^k)_{k =1, \ldots, N})$ be a  standard Brownian motions,
$\alpha$ be a control process adapted to the filtration generated by $(B, W^1, \ldots, W^N)$.
As above, we consider the controlled particle system $(X^{\alpha,N,k}, Z^{\alpha,N,k})_{k\le N}$ defined by
\begin{align*}
\d X^{\alpha,N,k}_t
=~&
\Big( b \big(t, X^{\alpha,N,k}_t, \alpha_t \big) - \sigma_0 h \big(t, X^{\alpha,N,k}_t \big) \Big) \d t
+
\sigma(t, X^{\alpha,N,k}_t, \alpha_t) \d W^k_t
+
\sigma_0(t, X^{\alpha,N,k}_t) \d B_t, \\
\d Z^{\alpha,N,k}_t
=~&
h(t, X^{\alpha,N,k}_t) Z^{\alpha,N,k}_t \d B_t,
\end{align*}
{with initial condition $X^{\alpha,N,k}_0 = \xi_k \sim \nu_k$ for each $k \ge 1$.}
Recall that $\Ac_N$ denotes the space of all control processes adapted to the filtration $\F^N$ generated by $(\xi_1, \ldots, \xi_N, W^1, \ldots, W^N, B)$.
The approximation of \eqref{eq:V_partial}  is then given by
\begin{equation} \label{eq:def_VNP}
V^N_P
:=
\sup_{\alpha \in \Ac_N}
\E \Big[
\frac1N \sum_{k=1}^N
\int_0^T Z^{\alpha,N,k}_t L \big(t, X^{\alpha,N,k}_t, \alpha_t \big) dt + Z^{\alpha,N,k}_T g \big(X^{\alpha,N,k}_T \big)
\Big].
\end{equation}

The following convergence result is   an immediate consequence of Theorem \ref{thm:weak_cvg}.

\begin{proposition}
Assume that $b, \sigma, \sigma_0, h$ and $\sigma_0 h$ satisfy the conditions in Assumption \ref{assum:main}.$(\mathrm{(i)}$,
and that $x\in \R^d\mapsto L(\cdot,x,\cdot)$ and $x\in \R^d\mapsto g(x)$ are continuous and have at most linear growth in $x$, {uniformly in their other arguments}.
Assume in addition that $\nu_k \in \Pc_p(\R^d)$ for all $k \ge 0$ and $\Wc_p \big(N^{-1}\sum_{k=1}^{N} \nu_k , \nu_0 \big) {\longrightarrow}0$ as $N\longrightarrow \infty$, for some $p > 2$.
Then
$$
V^N_P \longrightarrow V_P,
~\mbox{as}~
N \longrightarrow \infty.
$$
\end{proposition}
 
\begin{remark}
The controlled particles problem $V^N_P$ in \eqref{eq:def_VNP} is a time-consistent problem,
and hence can be solved by the dynamic programming approach, using a time discretization.
When $N$ is large, it is a high-dimensional control problem but one can rely on  machine learning algorithms.
Moreover, one can exploit the fact that the problem is symmetric to improve the algorithm.
We will explore this in a future project.
\end{remark}
\begin{remark}
The   formulation   \eqref{eq:def_VNP} does not fall in the BSDE setting, so that our results do not apply to provide a convergence rate.
We leave this for future researches.
\end{remark}

\appendix

\section{Appendix: Proof of Theorem \ref{thm:weak_cvg}}
\label{sec:proof_weak_cvg}
 
This Appendix is devoted to the proof of   Theorem \ref{thm:weak_cvg}. As a first step, we rewrite the control problems associated to $V$ and $V^N$ on a suitable canonical space, so as to be able to rely on weak compactness arguments.
 
\subsection{The canonical space formulations}
 
Let us denote by $\M$ the space of all positive Borel measures $q$ on $[0,T] \x A$ such that the marginal distribution on $[0,T]$ is the Lebesgue measure, so that one can write $q(\d t, \d a) = q_t(\d a) \d t$, where $(q_t(da))_{t\le T}$ is a Borel measurable kernel from $[0,T]$ to $\Pc(A)$.
Let us further denote by $\M_0 \subset \M$ the subset of all $q \in \M$ such that $q(\d t, \d a) = \delta_{\psi(t)} (\d a)$ for some Borel measurable function $\psi: [0,T] \longrightarrow A$.
 
\paragraph{Canonical space $\Omh$.}
Let us introduce a first canonical space 
$$
\Omh := \Cc^d \x \Cc^d \x \M \x \Cc^d,
$$
which is equipped with the canonical element $\big(\Xh, \Yh,\Lambdah, \Wh\big)$, the Borel $\sigma$--algebra $\widehat{\Fc}:=\Bc(\Omh)$, and the canonical filtration $\Fh:= \big(\widehat{\Fc}_t \big)_{t\le T}$ defined by
\[
\Fch_t
:=
\sigma \Big(\big(\Xh_s, \Yh_s, \Lambdah([0,s] \x D), \Wh_s\big): D \in \Bc(A),\; s \in[0, t] \Big),\; t\in[0,T].
\]
Notice that one can choose a version of the disintegration $\Lambdah(\mathrm{d}t,\mathrm{d}a)=\Lambdah_t(\mathrm{d}a)\mathrm{d}t$
such that $\Lambdah$ is a $\Pc(A)$--valued, $\Fh$--predictable process (see e.g. \cite[Lemma 3.2.]{LackerLimit}).
For all $\varphi \in C^2_b(\R^{d + d})$ and $(t,\xb,\yb,\wb, \nu, a)\in[0,T]\times\Cc^d \times\Cc^d \times\Cc^d\times\Pc(\Cc^d)\times A$,
we define
{ 
\begin{align} \label{eq:conditionnal_generator}
\widehat \Lc_t \varphi \big( \xb, \yb, \wb, \nu, a \big)
:=
\hat b(t, \xb, \nu, a)\cdot D\varphi(\yb_t,\wb_t)
+
\frac{1}{2}\mathrm{Tr}\big[ \hat a(t, \xb, \nu, a) D^2 \varphi(\yb_t,\wb_t)\big],
\end{align}
where
\[
\hat b \big(t,\xb, \nu,a \big)
:=
\begin{pmatrix}
b(t,\xb,\nu,a) \\ 0_d
\end{pmatrix},\;
\hat a \big(t, \xb, \nu, a \big)
:=
\begin{pmatrix}
\sigma(t,\xb,\nu,a) \\
\mathrm{I}_{d}
\end{pmatrix}
\begin{pmatrix}
\sigma(t,\xb,\nu,a) \\
\mathrm{I}_{d}
\end{pmatrix}^{\top}.
\]
}
Then, given a family $(\nu_t)_{t \le T}$ of probability measures in {$\Pc(\Cc^d)$} such that $ [0,T]\ni t \longmapsto \nu_t \in \Pc(\Cc^d)$ is Borel measurable, and $\varphi \in C^2_b(\R^{n + d})$, one can   define 
\begin{equation}\label{eq:Mvarphi}
\widehat{S}^{\varphi, \nu}_t
:=
\varphi \big(\Yh_t, \Wh_t \big)-\varphi(\Yh_0, \Wh_0)
-
{ \int_{[0,t]\x A}} \widehat \Lc_s \varphi\big(\Xh, \Yh, \Wh, \nu_s, a \big) \Lambdah_s(\mathrm{d}a)\mathrm{d}s,\; t\le T,
\end{equation}
where, for a Borel function $\phi:[0,T] \to \R,$ we write $\int_0^{\cdot} \phi(s)\mathrm{d}s:=\int_0^\cdot \phi^{+}(s)\mathrm{d}s-\int_0^\cdot \phi^{-}(s)\mathrm{d}s$ with the convention $\infty - \infty = -\infty.$

\paragraph{Canonical space $\Omb$.}
Next, we introduce a second canonical space  
$$
\Omb := \Cc^d \x \Cc^d \x \M \x \Cc^d \x \Cc^d \x \Pc \big(\Omh \big),
$$
which is equipped with the canonical element $(X, Y, \Lambda, W, B, \muh)$, and the Borel $\sigma$--algebra $\Fcb := \Bc(\Omb)$.
Based on the canonical element $\muh$, we define the processes $(\mu_t)_{t\le T}$, and $(\muh_t)_{t\le T}$ on $\Omb$ by
\begin{equation} \label{eq:muh2mu}
\mu_t:= \muh \circ \big( \Xh_{t \wedge \cdot}\big)^{-1},
\;\mbox{and}\;
\muh_t:=\muh \circ \big( \Xh_{t \wedge \cdot},\Yh_{t \wedge \cdot},\Lambdah^t,\Wh\big)^{-1},\; t\le T,
\end{equation}
in which {$\Lambdah^t  := \Lambda(\cdot \cap [0,t],\cdot) + \delta_{a_0}(\d a) \d s \1_{(t, T]}$}.
Next, we introduce two filtrations $\overline\F:=(\overline\Fc_t)_{t\le T}$ and $\Gb:=(\Gcb_t)_{t\le T}$ on $(\Omb,\Fcb)$, with
\[
\overline \Fc_t
:=
\sigma \Big((X_s, Y_s, \Lambda([0,s] \x D),W_s,B_s, \langle \muh_s, \phi \rangle): D \in \Bc(A), \phi \in C_b(\Cc^d \x \Cc^d \x \M \x \Cc^d),\; s \in[0,t] \Big),
\]
where $C_b(\Cc^d \x \Cc^d \x \M \x \Cc^d)$ denotes the space of bounded and continuous real-valued functions on $\Cc^d \x \Cc^d \x \M \x \Cc^d$, and
\begin{equation} \label{eq:def_Gc}
\Gcb_t
:=
\sigma \Big(\Lambda([0,s] \x D), B_s, \langle \muh_s, \phi \rangle: D \in \Bc(A), \phi \in C_b(\Cc^d \x \Cc^d \x \M \x \Cc^d), s\in[0,t] \Big), \; t\le T.
\end{equation}
Given $\varphi\in C^2_b(\R^{d+d+d +d})$, we set 
\begin{align} \label{eq:first_generator}
\overline \Lc_t \varphi \big( \xb,\yb,\wb,\bb, \nub, a \big)
:=&
\bar b(t, \xb, \nub, a)\cdot D \varphi(\xb_t,\yb_t,\wb_t,\bb_t) \nonumber \\
&+
\frac{1}{2}\mathrm{Tr}\big[ \bar a(t, \xb, \nub, a) D^2 \varphi(\xb_t,\yb_t,\wb_t,\bb_t)\big],
\end{align}
where
\begin{align*}
\bar b \big(t,\xb, \wb,\bb, \nub,a \big)
&:=\begin{pmatrix}
b(t,\xb, \nub, a)\\
b(t,\xb, \nub, a)\\
0_d\\
0_d
\end{pmatrix}, \\
\bar a \big(t,\xb,\wb, \bb, \nub,a \big)
&:=
\begin{pmatrix}
\sigma(t,\xb, \nub, a) & \sigma_0(t,\xb, \nub, a) \\
\sigma(t,\xb, \nub, a) & 0_{n \x d} \\
\mathrm{I}_{d \x d} & 0_{d \x d} \\
0_{d \x d} & \mathrm{I}_{d \x d}
\end{pmatrix}
\begin{pmatrix}
\sigma(t,\xb, \nub, a) & \sigma_0(t,\xb, \nub, a) \\
\sigma(t,\xb, \nub, a) & 0_{n \x d} \\
\mathrm{I}_{d \x d} & 0_{d \x d} \\
0_{d \x d} & \mathrm{I}_{d \x d}
\end{pmatrix}^{\top}.
\end{align*}
for $(t,\xb,\yb, \wb,\bb, \nub,a)\in[0,T]\times\Cc^d \x \Cc^d \x \Cc^d \x \Cc^d\x\Pc(\Cc^d\times A)\x A$, 
which leads to the definition of a continuous $\Fb$--adapted process  
\begin{align} \label{eq:associate-martingale}
\Sb^{\varphi}_t
:=
\varphi(X_t, Y_t, W_t, B_t)
-
\int_{[0,t]\x A} \overline \Lc_s \varphi \big(X_s, Y_s, W_s, B_s, \mub_s, a \big) \Lambda_s(\mathrm{d}a) \mathrm{d}s,\; t\in[0,T].
\end{align}

\begin{definition} \label{def:admissible_ctrl_rule}
Fix $\nu \in \Pc_p(\R^d)$ { for some $p \ge 2$.} A probability $\Pb$ on $(\Omb, \Fcb)$ is a relaxed control rule with initial condition $\nu$ if
\begin{itemize}
\item[$(i)$] $\Pb \big[X_0=Y_0,\; W_0=0,\; B_0=0 \big]=1$, $\Pb \circ X_0^{-1} = \nu$ and  $\E^{\Pb} \big[ \|X\|^p+{\int}_{[0,T]\x A} \big( \rho(a_0, a) \big)^p \Lambda_t(\mathrm{d}a) \mathrm{d}t \big] < \infty;$
 
\item[$(ii)$] The pair $(X_0,W)$ is independent of $\Gcb_T$ under $\Pb$, {and, with $(\Pb^{\Gcb_T}_{\omb})_{\omb \in \Omb}$ being a family of regular conditional probability distributions of $\Pb$ given $\Gcb_T$,}
\begin{equation} \label{eq:muh_property}
\muh_t (\omb)
=
\Pb^{\Gcb_T}_{\omb} \circ (X_{t \wedge \cdot},Y_{t \wedge \cdot}, \Lambda^t, W)^{-1},
\;\mbox{\rm for}\;\Pb\mbox{\rm--a.e.}\;\omb \in \Omb
\end{equation}
for all $t\le T$;
 
\item[$(iii)$] The process $ \Sb^{\varphi}$ is an $(\Fb, \Pb)$--martingale for all $\varphi \in C^2_b \big(\R^d \x \R^d \x \R^d \x \R^{d} \big)$; 
\item[$(iv)$] For $\Pb\;\mbox{\rm--a.e.}\;\omb \in \Omb$,
the process $ \widehat{S}^{\varphi, \mu(\omb)} $
is an $\big(\Fh,\muh(\omb) \big)$--martingale for each $\varphi \in C^2_b(\R^d \x \R^d)$.
\end{itemize}
\end{definition}

{Given ${\nu_0} \in \Pc_p(\R^d)$, one can then define the associated admissible control rules}
\[
\Pcb_R
:=
\big\{ \mbox{All relaxed control rules}\; \Pb \;\mbox{with initial condition}\; {\nu_0} \big\}.
\]
 
\begin{remark}
$\mathrm{((i)}$ Under Assumption \ref{assum:main} and the integrability condition in Definition \ref{def:admissible_ctrl_rule}.$(i)$,
the process $\Sb^{\varphi}$ is $\Pb$-square integrable for $\varphi \in C^2_b \big(\R^d \x \R^d \x \R^d \x \R^d \big)$.
Then, it does not change the definition of the admissible control rule if one changes {\ref{def:admissible_ctrl_rule}}.$(iii)$ to
``$ \Sb^{\varphi} $ is an $(\Fb, \Pb)$--{\bf local martingale} for all $\varphi \in C^2_b \big(\R^d \x \R^d \x \R^d \x \R^d \big)$.''
\vspace{0.5em}
 
\noindent $\mathrm{(ii)}$ Since $\muh_t$ is $\Gcb_t$--measurable, it follows that \eqref{eq:muh_property} is equivalent to
\begin{align} \label{eq:muh_property-details}
{\muh_t (\omb)=\Pb^{\Gcb_t}_{\omb} \circ (X_{t \wedge \cdot},Y_{t \wedge \cdot}, \Lambda^t,W)^{-1}} ,\;\mbox{\rm for}\;\Pb\mbox{\rm--a.e.}\;\omb \in \Omb,
\end{align}
{ with $(\Pb^{\Gcb_t}_{\omb})_{\omb \in \Omb}$ being a family of regular conditional probability distributions of $\Pb$ given $\Gcb_t$,}
\end{remark}

{Let us denote by $\Lc_0[A]$ the set of all Borel functions $\phi: [0,T] \x \Cc^d \longrightarrow A$, and further define}
\begin{equation} \label{eq:redef_VW}
\Pcb_W
:=
\big\{ \Pb \in \Pcb_R :\Pb \big[ \Lambda \in \M_0 \big]=1 \big\},
\end{equation}
and\
\begin{equation} \label{eq:redef_VS}
\Pcb_S
:=
\Big\{
\Pb \in \Pcb_W :\exists\; \phi \in \Lc_0[A],
~\Pb\big[\Lambda_t(\mathrm{d}a)\mathrm{d}t=\delta_{\phi(t, B_{t \wedge \cdot})}(\mathrm{d}a)\mathrm{d}t \big]=1
\Big\}.
\end{equation}
Finally, we define $J(\Pb)$ for $\Pb \in \Omb$ by
\begin{equation} \label{eq:def_VS-canon}
J \big(\Pb\big)
:=
\E^{\Pb} \bigg[\int_{[0,T]\x A} L\big(t, X_{t\wedge\cdot}, \mu_t, a \big) \Lambda_t(\mathrm{d}a) \mathrm{d}t +g\big( X_{T\wedge\cdot}, \mu_T \big) \bigg].
\end{equation}
 
\begin{remark}
Comparing to the canonical space formulation of the McKean-Vlasov optimal control problem in \cite{DjeteApprox}, the main difference is that $\Lambda^t$ is $\Gc_t$-measurable as defined in \eqref{eq:def_Gc}.
Consequently, given $\Pb \in \Pcb_R$, for $\Pb$ \mbox{-a.e.} $\omb \in \Omb$, $\Lambdah$ is $\mu(\omb)$ \mbox{-a.s.~equal to a deterministic measure}.
\end{remark}

For each {$\alpha \in \Ac_0$,} we define
$$
Y^{\alpha}_t
:=
X^\alpha_t-\int_0^t \sigma_0(s, X^\alpha_{s \wedge \cdot}, \mub^\alpha_s,\alpha_s ) \mathrm{d}B_s,\; t\in[0,T],\;
\Lambda^\alpha_t(\mathrm{d}a)\mathrm{d}t
:=
\delta_{\alpha_t}(\mathrm{d}a)\mathrm{d}t,
$$
and
$$
\muh^\alpha
:=
\Lc^{\P} \big( X^\alpha, Y^\alpha,\Lambda^\alpha,W \big| \Gcb_T \big).
$$

{This construction leads to equivalent definitions of $V$ and $V_W$ in terms of control problems set on $(\bar \Omega,\bar \Fc)$.}

\begin{proposition} \label{prop:weak_ctrl_rule} 
{In the context of Theorem \ref{thm:weak_cvg},} 
\begin{equation} \label{eq:def_PcbS}
\Pcb_S
=
\big\{
\P \circ \big(X^\alpha,Y^\alpha,\Lambda^\alpha,W, B, \muh^\alpha \big)^{-1}:\alpha \in \Ac_0
\big\},
\end{equation}
and 
$$
\Pcb_W
=
\big\{
\P^{\gamma} \circ \big(X^\gamma,Y^\gamma,\Lambda^\gamma,W^{\gamma}, B^{\gamma}, \muh^\gamma \big)^{-1}
:\gamma \in \Gamma \big\},
$$
so that {$V$ and $V_W$ can be equivalently formulated on the canonical space as follows:}
$$
V = \sup_{\Pb \in \Pcb_S} J\big(\Pb\big)
~\mbox{and}~
V_W = \sup_{\Pb \in \Pcb_W} J\big(\Pb\big).
$$
\end{proposition}
The proof is almost the same as the one of \cite[Proposition 2]{DjeteApprox}, and is omitted.
 
\vspace{0.5em}

We finally introduce the relaxed formulation of the control problem: 
\begin{equation} \label{eq:defVR}
	V_R := \sup_{\Pb \in \Pcb_R} J\big(\Pb\big).
\end{equation}

\subsection{Approximation results and proof of Theorem \ref{thm:weak_cvg}}

{In order to be able to apply standard propagation of chaos results for uncontrolled SDEs, we rely on piecewise constant controls, that serve as approximations of general controls.  We therefore introduce the space $\Ac_0^{\circ}$ of all elements $\alpha \in \Ac_0$ such that   $\alpha = \alpha_{t_i}$ on $ [t_i, t_{i+1})$, $i=0, \ldots, m-1$, for some partition $0 = t_0 < \ldots < t_m = T$. The corresponding set of controls is }
\begin{equation} \label{eq:PcbS0}
\Pcb^{\circ}_S
:=
\big\{
\P \circ \big(X^\alpha,Y^\alpha,\Lambda^\alpha,W, B, \muh^\alpha \big)^{-1}:\alpha \in \Ac^{\circ}_0
\big\}.
\end{equation}

{The first key results consists in showing that $\Pcb^{\circ}_S\subset \Pcb_S$ is dense in $\Pcb_W$, which in turn is dense in $\Pcb_R$. In particular, (i) below and Proposition \ref{prop:weak_ctrl_rule}  imply that $V=V_W$, the first assertion of Theorem \ref{thm:weak_cvg}.}

\begin{proposition} \label{prop:density_Pcb}
$\mathrm{(i)}$ Let Assumption \ref{assum:main} hold true and $\Pc(\Omb)$ be equipped with the $\Wc_2$ topology.
Then, the set $\Pcb^{\circ}_S$ is dense in $\Pcb_S$, and $\Pcb_S$ is dense in $\Pcb_W$. 

\vspace{0.5em}

\noindent $\mathrm{(ii)}$ Assume in addition that $\sigma_0$  is not controlled (i.e. $\sigma_0(t, \xb, \nu, a)$ is independent of $a$), $\nu_0 \in \Pc_p(\R^d)$ for some $p > 2$.
Then, the set $\Pcb_W$ is dense in $\Pcb_R$. 
Consequently,
{$$
	V = V_W = V_R.
$$
}
\end{proposition}
\begin{proof}
$\mathrm{(i)}$ { For the statement in   $\mathrm{(i)}$ of the Proposition, it is enough to prove that $\Pcb^{\circ}_S$ is dense in $\Pcb_W$, for which} 
we follow the main steps of \cite[Section 4.2.1]{DjeteApprox}.
However, since the control process is required to be adapted to the common noise filtration, the approximation and randomization procedures are slightly  different.
We explain here how the proof  of  \cite[Section 4.2.1]{DjeteApprox} should be modified.

\vspace{0.5em}

First, for each $m \ge 2$, let us define a partition $(t^m_i)_{0 \le i \le m}$ of $[0,T]$ by $t^m_i := iT/m$,
and then define $\eps_m := T/m$, $t^{\eps, m}_i := i \eps_m/m \in [0,t^m_1]$.
On the canonical space $\Omb$, we define $B^m $ by
$$
B^m_{s} =0, ~~\mbox{for}~s \in [0, t^m_1],
$$
and
$$
B^m_s = B_s - B_{t^m_1}, ~~\mbox{for}~s \in [t^m_1, T].
$$
In contrast to \cite[Section 4.2.1]{DjeteApprox} where one freezes both $B$ and $W$ on $[0, t^m_1]$,
here we only freeze $B$ on $[0, t^m_1]$.
\vspace{0.5em}
 
Let us further define two filtrations $\Fb^m = (\Fcb^m_t)_{t\le T}$ and $\Gb^m = (\Gcb^m_t)_{t\le T}$ by
$$
\Fcb^m_t := \sigma \big( X_{t \wedge \cdot}, Y_{t, \wedge \cdot}, W_{t \wedge \cdot}, B^m_{t \wedge \cdot}, \Lambda^t, \muh_t \big),
~~
\Gcb^m_t := \sigma \big( B^m_{t \wedge \cdot}, \Lambda^t, \muh_t \big), \;t\le T.
$$
Let $\Pb \in \Pcb_W$ be a fixed weak control measure.
Following \cite[Lemma 2]{DjeteApprox}, one can construct a sequence $(X^m, B^m, \muh^m, \alpha^m)_{m \ge 1}$ on $(\Omb, \Fcb, \Pb)$ satisfying the following properties:
 
\vspace{0.5em}
 
\noindent $\mathrm{(1)}$
The process $\alpha^m$ is $\G$-predictable and piecewise constant in the sense that $\alpha^m_s = \alpha^m_{t^m_i}$ for all $s \in [t^m_i, t^m_{i+1})$, $i = 1, \ldots, m-1$, {and}  $\alpha^m_s \equiv a_0$ for $s \in [0, \eps_m] = [0, t^m_1)$.
The process $X^m$ is a $\F$-adapted continuous process such that
\begin{equation} \label{eq:approx_weak}
\lim_{m \to \infty} \E^{\Pb} \Big[ \int_0^T \rho(\alpha^m_t, \alpha_t)^2 \d t + \| X- X^m \|^2 \Big] = 0.
\end{equation}
\noindent $\mathrm{(2)}$ The process $X^m$ is the unique solution, under $\Pb$,  to the SDE
$$
X^m = X_0 + \int_0^\cdot b \big(r, X^m_{r \wedge \cdot}, \mu^m_r, \alpha^m_r \big) \d r
+ \int_0^\cdot \sigma \big(r, X^m_{r \wedge \cdot}, \mu^m_r, \alpha^m_r \big) \d W_r
+ \int_0^\cdot \sigma_0 \big(r, X^m_{r \wedge \cdot}, \mu^m_r, \alpha^m_r \big) \d B^m_r,
$$
with $\E \big[ \|X^{{m}}\|^2 \big] < \infty$ and $\mu^m_t = \Lc(X^m_t | \Gc^m_t)$ for $t\le T$.
\vspace{0.5em}
 
\noindent $\mathrm{(3)}$ Given $\Lambda^m_t (\d a) \d t := \delta_{\alpha^m_t}(\d a) \d t$,
$$
Y^m_t := X^m_t - \int_0^t \sigma_0 \big(r, X^m_{r \wedge \cdot}, \mu^m_r, \alpha^m_r \big) \d B^m_r,
~~
\muh^m_t := \Lc^{\Pb} \big( X^m_{t \wedge \cdot}, Y^m_{t \wedge \cdot}, (\Lambda^m)^t, W^m \big| \Gc^m_t \big),
~t\le T,
$$
then it holds that, for any $t\le T$,
$$
\muh^m_t = \muh^m_T \circ \big(\Xh_{t \wedge \cdot}, \Yh_{t \wedge \cdot}, \Lambdah^t, \Wh \big)^{-1}
~~\mbox{and}~~
\muh^m_t := \Lc^{\Pb} \big( X^m_{t \wedge \cdot}, Y^m_{t \wedge \cdot}, (\Lambda^m)^t, W^m \big| B^m, \Lambda^m, \muh^m \big).
$$
Next, we consider $(W^m_k {=(W^m_{k,t})_{t \le T}}, B^m_k {=(B^m_{k,t})_{t \le T}}, \alpha^m_k, \muh^m_k)_{k = 1, \ldots, m}$, 
where for each $k= 1, \ldots, m$ {and for all $t \le T$},
$$
W^m_k := W_{(t^m_{k-1} \vee t) \wedge t^m_k} - W_{t^m_{k-1}},
~~~
B^m_k := B^m_{(t^m_{k-1} \vee t) \wedge t^m_k} - B^m_{t^m_{k-1}},
~~~
\alpha^m_k := \alpha^m_{t^m_k},
~~~
\muh^m_k := \muh^m_{t^m_{k}}.
$$
Recall also that $\eps_m {=} T/m$ and $t^{\eps, m}_i {=} i \eps_m/m$, 
so that $\Delta B^{m,\eps}_i := B_{t^{m,\eps}_i} - B_{t^{m,\eps}_{i-1}}$ depends only on the increment of $B$ on $[t^{m,\eps}_{i-1}, t^{m,\eps}_i] \subset [0, t^m_1]$,
$B^m_1 \equiv 0$ and for $k\ge 2$, $B^m_k$ depends only on the increment of $B$ on $[t^m_{k-1}, t^m_k] \subset [t_1, T]$.
Therefore, $\Delta B^{m,\eps}_i$
is independent of $(W^m_k, B^m_k, \alpha^m_k, \muh^m_k)_{k = 1, \ldots, m}$. 
Using similar arguments as in the proof of \cite[Lemma 3]{DjeteApprox} (or \cite[Section 4]{ElKarouiTan} for a simpler context),
there exists functions $(G^m_{k,1}, G^m_{k,2})$ such that,
with
$$
\big( \gamma^m_k, \nuh^m_k \big) :=
\big( G^m_{k,1}, G^m_{k,2} \big) \big( (B^m_i)_{i=1}^k, (\alpha^m_i)_{i=1}^{k-1}, (\muh^m_i)_{i=1}^{k-1}, ~\Delta B^{m,\eps}_k \big),
~~k=1, \ldots, m,
$$
one has
\begin{equation} \label{eq:law_equiv}
\Pb \circ \big( \big(W^m_k, B^m_k, \alpha^m_k, \muh^m_k \big)_{k=1, \ldots, m} \big)^{-1}
~=~
\Pb \circ \big( \big(W^m_{{k}}, B^m_k, \gamma^m_k, \nuh^m_k \big)_{k=1, \ldots, m} \big)^{-1}.
\end{equation}
Let us define $(\Xt^m, \Yt^m)$ by
$$
\Xt^m_t
= X_0
+ \int_0^t \!\!\! b \big(r, \Xt^m_{r \wedge \cdot}, \nu^m_r, \gamma^m_r \big) \d r
+ \int_0^t \!\!\! \sigma \big(r, \Xt^m_{r \wedge \cdot}, \nu^m_r, \gamma^m_r \big) \d W_r
+ \int_0^t \!\!\! \sigma_0 \big(r, \Xt^m_{r \wedge \cdot}, \nu^m_r, \gamma^m_r \big) \d B^m_r,
$$
and
$$
\Yt^m_t := {\Xt^m_t}- \int_0^t \!\!\! \sigma_0 \big(r, \Xt^m_{r \wedge \cdot}, \nu^m_r, \gamma^m_r \big) \d B^m_r,
$$
so that it is easy to check that $\gamma^m$ is $\F^B$-adapted, and
$$
\nuh^m_t = \Lc\big( \big( \Xt^m, \Yt^m, \Lambda^{\gamma^m,t}, W \big) \big| B \big).
$$
Finally, we define $(\Xt^{m'}, \Yt^{m'}, \muh^{m'})$ by
$$
\Xt^{m'}_t
= X_0
+ \int_0^t \!\!\! b \big(r, \Xt^{m'}_{r \wedge \cdot}, \mu^{m'}_r, \gamma^{m}_r \big) \d r
+ \int_0^t \!\!\! \sigma \big(r, \Xt^{m'}_{r \wedge \cdot}, \mu^{m'}_r, \gamma^{m}_r \big) \d W_r
+ \int_0^t \!\!\! \sigma_0 \big(r, \Xt^{m'}_{r \wedge \cdot}, \mu^{m'}_r, \gamma^{m}_r \big) \d B_r,
$$
with {$\mu^{m'}_t = \Lc(X_{t \wedge \cdot} | B)$}, and
$$
\Yt^{m'}_t :={\Xt^{m'}_t}- \int_0^t \!\!\! \sigma_0 \big(r, \Xt^{m'}_{r \wedge \cdot}, \mu^{m'}_r, \gamma^{m'}_r \big) \d B^m_r,
$$
and
$$
{ \muh^{m'}_t = \Lc\big( \big( \Xt^{m'}_{t \wedge \cdot}, \Yt^{m'}_{t \wedge \cdot}, \Lambda^{\gamma^m,t}, W \big) \big| B \big).}
$$
As in \cite[Lemma 2]{DjeteApprox}, one can conclude that
$$
\P \circ \big( \Xt^{m'}, \Yt^{m'}, \Lambda^{\gamma^m}, W, B, \muh^{m'} \big)
~\in~
\Pc^{\circ}_S,
$$
and that
$$
\lim_{m \to \infty} \E^{\Pb}
\Big[
  \| \Xt^m - \Xt^{m'}\|^2
+
\Wc_2( \nuh^m, \muh^{m'})^2
\Big]
= 0.
$$
Combining with \eqref{eq:approx_weak} and \eqref{eq:law_equiv}, this proves that {$\Pcb^{\circ}_S$ is dense in $\Pcb_W$.}

\vspace{0.5em}

\noindent $\mathrm{(ii)}$ To prove that $\Pcb_W$ is dense in $\Pcb_R$, one can apply exactly the same proofs in \cite[Section 4.2.2]{DjeteApprox}\footnote{ In the published version, the coefficient function $\sigma_0$ is assumed to be a constant, while in a longer version with detailed proofs on arXiv or in the PhD thesis \cite{DjeteThese}, the results hold still when the coefficient function $\sigma_0 (t, \xb, \nu, a)$ could be a function of $(t, \xb, \nu)$ but independent of $a$ as in the statement of Proposition \ref{prop:density_Pcb}.}.
In our context, as $\Lambda^t$ is adapted to the common noise filtration $\G$,
it follows that, for any $\Pb \in \Pcb_R$, for $\Pb$-a.e. $\omb \in \Omb$,
on the space $(\Omh, \Fch, { \muh(\omb)})$ 
the control measure $\Lambdah$ is equal to a deterministic measure a.s.
As a result, the projection arguments in \cite[Lemma 4]{DjeteApprox} leads to a sequence of deterministic process $\big( \widehat \alpha^m \big)_{m \ge 1}$ such that
$$
\delta_{ \widehat \alpha^m (\d a)} \d t
~\longrightarrow~
\Lambdah (\d a, \d t),
~\muh(\omb)\mbox{-a.s.}
$$
Consequently, in the  approximation sequence $(\Pb_m)_{m \ge 1}$ in \cite[Proposition 7]{DjeteApprox},
each $\Pb_m$ is a weak control in the sense of \cite{DjeteApprox} but with controls adapted to the common noise filtration $\G$,
and hence it is a weak control measure in the sense of   Section \ref{subsec:weak_cvg} and \eqref{eq:redef_VW}.
Then, one concludes the proof by \cite[Proposition 7]{DjeteApprox}.
\end{proof}

 We can now conclude the proof of  Theorem \ref{thm:weak_cvg} by showing that its second assertion holds true.
 
For any strong control process $\alpha \in \Ac_N$, let $(X^{\alpha, k})_{k=1, \ldots, N}$ be the unique strong solution to the SDE \eqref{eq:def_Xn},
we next define
$$
Y^{\alpha,k}_{\cdot}:=X^{\alpha,k}_{\cdot}-\int_0^\cdot \sigma_0 (s,X^{\alpha,k},\varphi^N_s,\alpha_s)\mathrm{d}B_s,
~~~
\mu^{X^N}_t
:=
\frac{1}{N} \sum_{k=1}^N \delta_{\big(X^{\alpha,k}_{t \wedge \cdot},Y^{\alpha,k}_{t \wedge \cdot},\delta_{\alpha_t}(\mathrm{d}a)\mathrm{d}t, W^k \big)},
$$
and then
\begin{equation} \label{eq:def_PN}
\Pb^{N, \alpha}
:=
\frac{1}{N} \sum_{k=1}^N \P^N \circ
\big(X^{\alpha,k},Y^{\alpha,k}, \delta_{\alpha_t}(\mathrm{d}a)\mathrm{d}t,W^k,B, \mu^{X^N} \big)^{-1}
\in
\Pc(\Omb).
\end{equation}

\begin{proposition}
{
Let Assumption \ref{assum:main} hold true. Suppose in addition that the coefficient function $\sigma_0$  is not controlled (i.e. $\sigma_0(\cdot, a)$ is independent of $a$), and, for some $p > 2$, $\nu_k \in \Pc_p(\R^d)$ for all $k \ge 0$ and $\Wc_p \big(N^{-1}\sum_{k=1}^{N} \nu^k , \nu \big)\underset{N\to\infty}{\longrightarrow}0$.

\vspace{0.5em}
 
\noindent $\mathrm{(i)}$ Let $\big(\Pb^N\big)_{N \ge 1}$ be given by $\Pb^N := \Pb^{N, \alpha_N}$,
where
$\alpha_N \in \Ac_N$ satisfies
\begin{equation} \label{eq:eps_optimal_ctrl}
J_N(\alpha) \ge V_S^N - \varepsilon_N,
\;\mbox{\rm for all}\;N \ge 1,
\end{equation}
for a sequence $(\varepsilon_N)_{N \ge 1} \subset \R_+$ satisfying $\lim_{N\to\infty} \varepsilon_N=0$.
Then, the sequence $(\Pb^N)_{N \ge 1}$ is relatively compact under $\Wc_2$,
and for any converging subsequence $\big(\Pb^{N_m}\big)_{m \ge 1}$, one has
$$
\lim_{m \to \infty} \Wc_2 \big( \Pb^{N_{m}}, \Pb^\infty \big)=0,
\;\mbox{\rm for some}
\;\Pb^\infty \in \Pcb_R.
$$
\noindent $\mathrm{(ii)}$
For any $\Pb \in \Pcb_R$,
one can construct a sequence $(\Pb^N)_{N \ge 1}$ such that 
$\Pb^N := \P^{N,\alpha_N}$ for  some $\alpha_N \in \Ac_N$, for each $N\ge 1$, and  $\Wc_2 \big( \Pb^{N}, \Pb \big) \underset{N\to\infty}{\longrightarrow} 0$.

\vspace{0.5em}

\noindent $\mathrm{(iii)}$ Consequently,  
$$
\lim_{N \to \infty} V^N
~=~
V.
$$
}
\end{proposition}
\begin{proof}
$\mathrm{(i)}$ First, we can apply \cite[Proposition 10]{DjeteApprox} to deduce that $(\Pb^N)_{N \ge 1}$ is relatively compact w.r.t. $\Wc_2$.
Up to considering a subsequence, one can us assume that $\Pb^N \longrightarrow \Pb^{\infty}$ under $\Wc_2$, { for some measure $\Pb^{\infty}$ and we need to prove that $\Pb^{\infty} \in \Pcb_R$.}
\vspace{0.5em}
 
For this, one can follow the arguments in \cite[Section 4.3.2.2]{DjeteApprox}, which induces that
$$
\Pb^{\infty} \big[ \muh \circ (\Xh_0)^{-1} = \nu_0, X_0 - Y_0 = W_0 = B_0 = 0 \big] = 1,
~~
\E^{\Pb^{\infty}} \Big[ \big\| X \big\|^2 + \int_{[0,T]\x A} \rho(a, a_0)^2 \Lambda( \d a, \d t) \Big] < \infty.
$$
Now, we slightly adapt the arguments in \cite[Section 4.3.2.2]{DjeteApprox},
and observe that for any $\phi \in C_b(\Cc^d \x \Cc^d \x \M \x \Cc^d)$ and $\psi \in C_b(\Cc^d \x \M \x \x \Pc(\Omh))$
\begin{align*}
&\E^{\Pb^{\infty}}
\big[ \phi \big(X_{t \wedge \cdot}, Y_{t \wedge \cdot}, \Lambda^t,W \big) \psi \big(B, \Lambda, \muh \big) \big]
\\
=&
\lim_{N\to\infty} \E^{\Pb^N}
\big[ \phi \big(X_{t \wedge \cdot}, Y_{t \wedge \cdot}, \Lambda^t,W \big) \psi \big(B, \Lambda, \muh \big) \big] \\
=&
\lim_{N\to\infty} \frac{1}{N} \sum_{k=1}^N \E^{\P^N}
\Big[ \phi \big(X^k_{t \wedge \cdot}, Y^k_{t \wedge \cdot}, (\delta_{\alpha^{N}_s}(\d a) \d s)^t,W^k \big) \psi \Big(B, \delta_{\alpha^{N}_s}(\d a) \d s, {\mu^{X^N}_T} \Big) \Big]
\\
=&
\lim_{N\to\infty} \E^{\P^N}
\Big[ \E^{{\mu^{X^N}_T}} \big[ \phi \big(\Xh_{t \wedge \cdot}, \Yh_{t \wedge \cdot}, \Lambdah^t,\Wh \big) \big] \psi \Big(B, \delta_{\alpha^{N}_s}(\d a) \d s, {\mu^{X^N}_T} \Big) \Big]\\
=&
\E^{\Pb^{\infty}}
\Big[
\E^{{\hat \mu}}\big[ \phi \big( \Xh_{t \wedge \cdot}, \Yh_{t \wedge \cdot},\Lambdah^t, \Wh \big) \big]
\psi \big(B, \Lambda, \muh \big)
\Big],
\end{align*}
which implies that
$$
\muh_t(\omb) = \Lc^{\Pb}( X_{t \wedge \cdot}, Y_{t, \wedge \cdot}, \Lambda^t, W \big| \Gc_T \big) (\omb),
~\mbox{for}~
{\Pb^{\infty}}\mbox{-a.e.}~\omb \in \Omb,\; t\le T.
$$
In particular, the fact that $\Pb^{\infty}\big[ \muh \circ (\Xh_0)^{-1} = \nu_0\big] = 1$ implies that $\Lc^{\Pb^{\infty}}(X_0 | \Gc_T) = \nu_0$, $\Pb^{\infty}$-a.s. and hence $\Pb \circ X_0^{-1} = \nu_0$.

\vspace{0.5em}

Given the above, the rest of the argument to prove that $\Pb^{\infty}$ {belongs to $\Pcb_R(\nu)$} is exactly the same as in \cite[Section 4.3.2.2]{DjeteApprox}.
 
\vspace{0.5em}
\noindent $\mathrm{(ii)}$ 
Given $\Pb \in  \Pcb_R$, it follows from Proposition \ref{prop:density_Pcb} that one can approximate it in $\Wc_2$ by a sequence $(\P^m)_{m\ge 1}\subset \Pcb^{\circ}_S$, i.e.~corresponding to piecewise-constant controls $(\alpha^m)_{m\ge 1}$ associated to a sequence of partition $(t^m_i,i\le m)_{m\ge 1}$. By a further approximation argument, one can assume that   each  $\alpha^m_{t^m_i}$ is a continuous function of $B_{t^m_i \wedge \cdot}$.  For $m$ fixed, standard propagation of chaos results for uncontrolled McKean-Vlasov SDEs, see e.g.~\cite{Sznitman}, imply that $\Wc_2 \big( \Pb^{N,\alpha^m}, \Pb^m \big)  {\longrightarrow} 0$ as $N\to \infty$, recall \eqref{eq:def_PN}. As a result, one can find a sequence $(m_N)_{N\ge 1}$ such that $\Wc_2 \big( \Pb^{N,\alpha^{m_N}}, \Pb \big)  {\longrightarrow} 0$ as $N\to \infty$

\vspace{0.5em}

\noindent $\mathrm{(iii)}$ In view of $\mathrm{(i)-(ii)}$, it is easy to deduce that
$$
\limsup_{N \to \infty} V^N \le J(\Pb^{\infty}) \le V,
~~\mbox{and}~~
\liminf_{N \to \infty} V^N \ge \sup_{\Pb \in \Pcb_R} J(\Pb) = V,
$$
which concludes the proof.
\end{proof}
 
\begin{remark}
{
Using similar arguments, one can also deduce the continuity of the value function $V$ 
w.r.t. the initial distribution $\nu_0$.
More precisely, let us write $V(\nu_0)$ (resp. $V^N(\nu_1, \ldots, \nu_N)$) in place of $V$ (resp.~$V^N$) to emphasize its dependance on the initial distribution $\nu_0$ (resp.~$(\nu_1, \ldots, \nu_N)$).
Then, for any sequence $(\nu_N)_{N \ge 1} \subset \Pc_p(\R^d)$, for some $p > 2$, and such that $\lim_{N \to \infty} \Wc_2(\nu_N , \nu_0) = 0$, one has
$$
\lim_{N \to \infty} V(\nu_N) = V(\nu_0).
$$
{In particular}, 
\begin{align*}
& \lim_{N \to \infty}
\bigg|
V^N(\nu_1, \ldots, \nu_N)  -V\bigg(\frac{1}{N}\sum_{k=1}^N \nu_k \bigg)
\bigg| \\
\le &
\lim_{N \to \infty}
\bigg|
V^N (\nu_1, \ldots, \nu_N) -V\big( \nu \big)
\bigg|
+
\lim_{N \to \infty}
\bigg|
V\bigg(\frac{1}{N}\sum_{k=1}^N \nu_{{k}} \bigg) - V \big( \nu \big)
\bigg| 
~=~
0
\end{align*}
{whenever  $\Wc_p \big(N^{-1}\sum_{k=1}^{N} \nu^k , \nu \big)\underset{N\to\infty}{\longrightarrow}0$.}}
\end{remark}

\end{document}